\documentclass[a4paper, 10pt]{article}
\usepackage{srcltx,graphicx}
\usepackage{amsmath, amssymb, amsthm}
\usepackage{color, xcolor}
\usepackage{array}
\usepackage{lscape}
\usepackage{algorithm}      
\usepackage[noend]{algpseudocode}
\usepackage{multirow}
\usepackage{psfrag}
\usepackage{pifont}
\usepackage{bbding}
\usepackage[pagebackref]{hyperref}
\usepackage{makeidx}
\usepackage{caption}
\usepackage{subcaption} 
\usepackage{afterpage}
\usepackage[section]{placeins}
\usepackage{overpic}

\newtheorem{observacion}{Remark}

\definecolor{myred}{rgb}{1,0,0}

\newtheorem{theorem}{Theorem}

{\theoremstyle{remark} \newtheorem{remark}{Remark}}

\hypersetup{linktoc = all}                
\hypersetup{hidelinks}
\hypersetup{bookmarksnumbered}
\newcommand \red {\color{red}}

\newcommand\diag{\mathrm{diag}}

\numberwithin{equation}{section}

\graphicspath{{images/}}

\newcommand\showfigures[1]{}

\newcommand \be { \begin{equation}}
\newcommand \ee { \end{equation}}

\begin{document}

\title{Steady States and Well-balanced Schemes for Shallow Water Moment Equations with Topography
\newline
\textit{}
%

\author{
Julian Koellermeier\thanks{Department of Computer Science, KU Leuven, 3001 Leuven, Belgium, 
    email: {\tt julian.koellermeier@kuleuven.be}},~
Ernesto Pimentel-Garc\'ia\thanks{Dpto. An\'alisis Matem\'atico. Universidad de M\'alaga, 14071 M\'alaga, Spain}
}
\date{\today}
}

\maketitle

\begin{abstract}
In this paper, we investigate steady states of shallow water moment equations including bottom topographies. We derive a new hyperbolic shallow water moment model based on linearized moment equations that allows for a simple assessment of the steady states. After proving hyperbolicity of the new model, the steady states are fully identified. A well-balanced scheme is adopted to the specific structure of the new model and allows to preserve the steady states in numerical simulations. 
\end{abstract} 
{\bf Keywords}: Shallow Water Equations, hyperbolic moment equations, well-balanced, steady states\\


\setcounter{tocdepth}{2}  


\section{Introduction}
\label{sec:Intro}
Applications of shallow flows can be found in many scientific fields, e.g., in hydrodynamics \cite{schijf1953theoretical} or granular flows \cite{craster2009dynamics}. An important class of problems considers changing topographies, for example related to snow avalanches \cite{christen2010ramms} or sediment transport \cite{Garres2020}. The main assumption for the widely used Shallow Water Equations (SWE) is that the horizontal velocity profile is constant along the vertical axis from the bottom to the surface. However, this assumption quickly brakes down for more complex flows that yield velocity variations. This is true in practically all applications of shallow flows and especially in presence of friction terms. But even in typical tsunami or dam break situations, the assumption of constant velocity profiles is often violated, see \cite{Koellermeier2020}. A new model that takes into account horizontal velocity changes over the vertical direction was developed in \cite{kowalski2018moment} based on an expansion of the velocity profile in polynomial basis functions modeling the deviation from a constant velocity profile. The resulting Shallow Water Moment Equations (SWME) are more accurate the more basis functions are considered. Despite the success for simple test cases, the model lacks hyperbolicity, which was studied in detail in \cite{Koellermeier2020}. In the same paper, a new model called Hyperbolic Shallow Water Moment Equations (HWSME) and a second variant called the $\beta$-HSWME were derived. The models are essentially based on a linearization of the original SWME model around linear velocity profiles, i.e., all contributions of coefficients of higher order basis functions are neglected. In \cite{Koellermeier2020} the eigenvalues of these models were analyzed and the first numerical tests confirmed that the models yield similar accuracy as the SWME models with additional guaranteed hyperbolicity. 

While the numerical tests in \cite{kowalski2018moment,Koellermeier2020} included standard friction terms for a Newtonian fluid, only a flat bottom topography was considered. This is obviously a strong simplification and bottom topographies need to be taken into account as has been done for the SWE in a numerical and analytical way, see \cite{alcrudo2001exact,bernetti2008exact,rosatti2010riemann,lefloch2007riemann} and the references therein. In the context of varying bottom topographies, it is paramount to consider steady states of the models because any numerical simulation should be able to exactly preserve steady states when present. Otherwise, numerical solutions starting from steady state initial conditions would lead to numerical artifacts or numerical instabilities. It is therefore important to first study the steady states of the models and then design tailored well-balanced numerical schemes, which means that the schemes preserve those steady states by balancing the topography source term and the numerical flux in the correct way so they cancel out. Since \cite{bermudez1994upwind}, the study and design of well-balanced numerical methods have been very active fields in the last years, see for instance \cite{audusse2004fast,berberich2019high,canestrelli2009well,castro2017well,castro2007well,noelle2006well,rebollo2003family,russo2008high}. In the context of path-conservative methods introduced in \cite{pares2006numerical}, the authors in \cite{castro2008well} and more recently in \cite{castro2020well} developed a strategy to obtain well-balanced high-order numerical methods for systems of balance laws. We will follow this strategy and apply it to a newly derived moment model.

In this paper, we investigate steady states of shallow water moment equations including bottom topographies and use this to derive a new first order and second order well-balanced numerical scheme for a new shallow water moment model. The analysis of the existing SWME, including the hyperbolic versions HSWME and $\beta$-HSWME, shows that steady states are difficult to access analytically and numerically, despite the simple case where the velocity profile is just a linear function of the vertical variable. Knowing about the problematic terms in the existing models, we derive a new model that is valid for small deviations from the constant velocity profile. For this model, we can neglect only the non-linear contributions of the basis coefficients while keeping the linear contributions of all coefficients. The model is thus called Shallow Water Linearized Moment Equations (SWLME). It is surprisingly simple, in the sense that it removes some coupling terms between the equations, but it keeps the overall structure even in the higher order equations. Subsequently, we prove hyperbolicity, analyze the eigenstructure, and show that the model yields more realistic propagation speeds than the previous models, while still being hyperbolic. Most importantly, the model allows for a concise characterization of its steady states with and without topography terms. The characterization of the steady states then allows to derive a potentially high-order well-balanced numerical scheme based on the possible steady states of the new model. We explicitly construct the first order and second order well-balanced scheme in this paper. The numerical schemes are tested extensively with a standard lake-at-rest test case, two subcritical stationary solutions, and a transcritical solution. In the end, we also present a test case comparing the new SWLME to the existing HSWME and $\beta$-HSWME models, to outline the good approximation properties of the new model despite its simplicity. 

The rest of the paper is organized as follows: In Section \ref{sec:vertical} we review the derivation of a vertically resolved shallow flow model that is employed to derive all the shallow water moment models presented in this paper. In the following sections we derive and analyze the standard Shallow Water Equations (SWE) (Section \ref{sec:SWE}), the extended Shallow Water Moment Equations (SWME) (Section \ref{sec:SWME}), and the new Shallow Water Linearized Moment Equations (SWLME) (Section \ref{sec:SWLME}) including their hyperbolicity, steady states, and Rankine-Hugoniot conditions including bottom topography. In Section \ref{sec:numerics}, we develop a first order and second order well-balanced numerical scheme for the special case of the shallow water models used in this paper. Numerical tests in Section \ref{sec:results} show the preservation of the steady states and allow for a comparison of the new SWLME model with respect to other existing models.

\section{Vertically resolved shallow flow model}
\label{sec:vertical}
In this paper, we are concerned with free-surface water flows in one horizontal direction. 
Modeling of free-surface flows starts with the incompressible Navier-Stokes equations, which model the evolution of the horizontal velocity $u$ in direction $x$ and the vertical velocity $w$ in direction $z$.
\begin{eqnarray}
    \label{e:NSE}
    \partial_x u + \partial_z w &=&0, \\
    \partial_t u + \partial_x u^2 + \partial_z uw &=& - \frac{1}{\rho} \partial_x p + \frac{1}{\rho} \partial_z \sigma_{xz} + g,
\end{eqnarray}
where $\rho$ is the density and $g$ the gravitation constant. The hydrostatic pressure in relation to the vertical position $z$ with respect to the surface $h+b$, where $b$ represents the bottom topography and $h$ is the water height, is given by
\begin{equation}
    p= (h+b - z) \rho g
\end{equation}
and the stress $\sigma_{xz}$ is modeled using the assumption of a Newtonian fluid with dynamic viscosity $\mu$, i.e.,
\begin{equation}
    \sigma_{xz}= \mu \partial_z u
\end{equation}
to close the system.

To allow for a more efficient representation of the horizontal velocity variation along the vertical axis, a mapping is introduced in \cite{kowalski2018moment}. This mapping shifts and scales the vertical variable, which is defined between the bottom at $z=b$ and the surface at $z=h+b$ according to the following transformation
\begin{equation}
    \label{e:mapping}
    \zeta = \frac{z - b}{h},
\end{equation}
where the denominator is precisely the water height $h$. The variable $\zeta$ is thus defined within the interval $[0,1]$.
According to the derivation in \cite{kowalski2018moment}, the following vertically-resolved system for the simulation of shallow flows is derived using the mapping from \eqref{e:mapping}
\begin{eqnarray}
    \partial_t h + \partial_x h u_m &=&0, \label{e:vertical1}\\
    \partial_t hu + \partial_x \left( hu^2 + \frac{g}{2} h^2 \right) + \partial_{\zeta} \left( hu \omega - \frac{1}{\rho} \sigma_{xz} \right) &=& -gh \partial_x b,\label{e:vertical2}
\end{eqnarray}
where $u_m$ is the mean velocity over the vertical $\zeta$-axis and the so-called vertical coupling term $\omega$ is given by
\begin{equation}
    \omega = \int_{0}^{\zeta} \left( \int_{0}^{1} \partial_x(hu)\left(\hat{\zeta}\right) \, d\hat{\zeta} - \partial_x(hu)\left(\widetilde{\zeta}\right) \right) \, d\widetilde{\zeta}.
\end{equation}

The following boundary conditions in the $\zeta$-direction are used:
\begin{eqnarray}
    \partial_{\zeta} u |_{\zeta=1} &=& 0,\\
    \partial_{\zeta} u |_{\zeta=0} &=& \frac{h}{\lambda} u|_{\zeta=0},
\end{eqnarray}
modeling a stress-free top surface and a slip condition at the bottom with slip length $\lambda$, see \cite{kowalski2018moment} for more details.

The system \eqref{e:vertical1}-\eqref{e:vertical2} is called \emph{vertically resolved system} in \cite{kowalski2018moment} as it includes the dependence on the vertical variable $\zeta$. This system is at the core of this paper as all the models are derived directly from it.



\section{Shallow Water Equations}
\label{sec:SWE}

Similar to the work in \cite{kowalski2018moment}, we will start with the simple Shallow Water Equations (SWE), which assume a constant velocity $u(t,x,\zeta) = u_m(t,x)$ over the whole vertical axis $\zeta$, see Figure \ref{fig:const_profile}. The dependency on the vertical variable $\zeta$ is then resolved by integrating over $\zeta \in [0,1]$ and using the constant velocity $u(t,x,\zeta) = u_m(t,x)$. It was shown in \cite{kowalski2018moment} that the vertically resolved system \eqref{e:vertical1}-\eqref{e:vertical2} then simplifies to the following set of equations called Shallow Water Equations (SWE)
\begin{equation}
    \partial_t
    \begin{pmatrix}
    h\\
    h u_m\\
    \end{pmatrix} +\partial_x
    \begin{pmatrix}
    h u_m\\
    h u_m^2 + \frac{1}{2}g h^2 \\
    \end{pmatrix} =
    \begin{pmatrix}
    0\\
     -g h \partial_x b\\
    \end{pmatrix}
    -\frac{\nu}{\lambda}
   \begin{pmatrix}
    0\\
    u_m\\
    \end{pmatrix}
\end{equation}

\begin{figure}[h]
		\begin{subfigure}{0.5\textwidth}
		    \centering
			\includegraphics[width=0.35\textwidth]{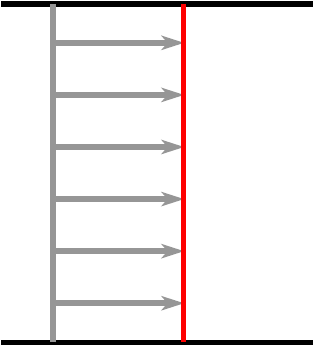}
			\caption{Constant velocity profile}
		    \label{fig:const_profile}
		\end{subfigure}
		\begin{subfigure}{0.5\textwidth}
		    \centering
			\includegraphics[width=0.35\textwidth]{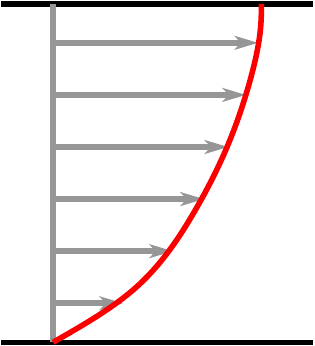}
			\caption{Varying velocity profile}
		    \label{fig:var_profile}
		\end{subfigure}
		\caption{Constant velocity ansatz of SWE model (a) and variable velocity ansatz of SWME model (b).}
		\label{fig:vel_profile}
\end{figure}
where $u_m=u_m(t,x)$ is the horizontal water velocity, $h=h(t,x)$ is the water height, $g$ is the gravitational constant (we later set it to $g=1$ in our simulations) the known function $b(x)$ is the bottom topography, and $\nu$ and $\lambda$ are the kinematic viscosity and the slip length, respectively.

In non-conservative matrix form, the model can be written as 
\begin{equation}
    \partial_t
    \begin{pmatrix}
    h\\
    h u_m\\
    \end{pmatrix} +
    \begin{pmatrix}
     0 & 1 \\
    -u_m^2 + g h & 2 u_m
    \end{pmatrix} \partial_x \begin{pmatrix}
    h\\
    h u_m\\
    \end{pmatrix}= 
   \begin{pmatrix}
    0\\
    -g h \partial_x b\\
    \end{pmatrix}
    -\frac{\nu}{\lambda}
   \begin{pmatrix}
    0\\
    u_m\\
    \end{pmatrix}.
\end{equation}

The eigenvalues of the left hand side transport matrix are the standard propagation speeds of the Shallow Water Equations
\begin{equation}
    \lambda_{1,2}=u_m \pm \sqrt{g h}.
\end{equation}

For flat bottom $\partial_x b=0$ and zero friction, the steady state fulfils
\begin{eqnarray}
    \partial_x \left(h u_m\right) &=& 0, \label{e:mass_eqn}\\
    \partial_x \left( h u_m^2+\frac{1}{2}g h^2\right) &=& 0,
\end{eqnarray}
so that the jump conditions (also called Rankine-Hugoniot conditions) from a given state $\left( h_0, h_0 u_{m,0} \right)$ to a state $\left( h, h u_m \right)$ can be derived by solving the system 
\begin{eqnarray}
    h u_m &=& h_0u_{m,0},\\
    h u_m^2+\frac{1}{2}g h^2 &=& h_0 u_{m,0}^2+\frac{1}{2}g h_0^2,
\end{eqnarray}
for which the solution is 
\begin{equation}
    \left(\frac{h}{h_0}\right) = -\frac{1}{2} + \frac{1}{2} \cdot \sqrt{1 + 8 Fr^2},
\end{equation}
where $Fr$ is the Froude number for the given state defined by 
\begin{equation}
    Fr = \frac{u_{m,0}}{\sqrt{g h_0}}.
\end{equation}

For a smooth frictionless flow including a bottom topography, the steady state momentum equation can be modified using the mass equation \eqref{e:mass_eqn} to
\begin{equation}
    \partial_x \left( \frac{1}{2}u_m^2 + g(h+b) \right) = 0.
\end{equation}
The steady state solution can thus be found using
\begin{eqnarray}
    h u_m &=& const,\\
    \frac{1}{2}u_m^2 + g(h+b) &=& const.
\end{eqnarray}

The SWE are widely used in simulations of water flows. However, the main deficiency is that the horizontal velocity $u$ is constant over the height by assumption. The model is thus not able to predict more complex flow phenomena.

\section{Shallow Water Moment Equations }
\label{sec:SWME}
For the Shallow Water Moment Equations (SWME) derived in \cite{kowalski2018moment}, the idea is to allow for a vertical variation of the water velocity profile. This is done by assuming the following ansatz for the velocity profile, see Figure \ref{fig:var_profile}:
\begin{equation}\label{expansion}
    u(t,x,\zeta)=u_m(t,x)+\sum_{j=1}^{N}\alpha_j(t,x)\phi_j(\zeta),
\end{equation}
where $u_m(t,x)$ is the mean horizontal velocity also used in the SWE in Section \ref{sec:SWE}, $\zeta$ is the scaled vertical coordinate \eqref{e:mapping}, $\alpha_j$ are coefficients, and $\phi_j$ are Legendre ansatz functions for $j=1, \ldots, N$ defined by
\begin{equation} \label{eq:defphi}
  \phi_j(\zeta) = \frac{1}{j!} \frac{d^j}{d\zeta^j} (\zeta - \zeta^2)^j.
\end{equation}
Note that the larger $N$, the more variation is allowed in vertical direction. Furthermore, the ansatz functions form a group of orthogonal basis functions as \cite{kowalski2018moment} 
$$
\int_0^1 \phi_m \phi_n d\zeta = \frac{1}{2n+1} \delta_{mn},
$$
with Kronecker delta $\delta_{m,n}$. 

The initial values for $u_m$ and $\alpha_j$ for $j=1,\ldots, N$ can be computed from some initial velocity profile $u(0,x,\zeta) = u_0(x,\zeta)$ by projecting the initial velocity profile to the basis functions $\phi_j$.
\begin{eqnarray}
    u_0(x,\zeta) &=& u_m(0,x)+\sum_{j=1}^{N}\alpha_j(0,x)\phi_j(\zeta) \\
    \int_0^1 u_0(x,\zeta) \phi_i (\zeta) \, d\zeta &=& \int_0^1 \left( u_m(0,x)+\sum_{j=1}^{N}\alpha_j(0,x)\phi_j(\zeta) \right) \phi_i (\zeta) \, d\zeta \\
    &=& u_m(0,x) \delta_{i,0} + \sum_{j=1}^{N}\alpha_j(0,x) \frac{1}{2i+1} \delta_{i,j},
\end{eqnarray}
which leads to the initial mean and coefficients
\begin{eqnarray}
    u_m(0,x) &=& \int_0^1 u_0(x,\zeta) \, d\zeta, \label{e:IC_u}\\
    \alpha_i(0,x) &=& (2i+1) )\int_0^1 u_0(x,\zeta) \phi_i (\zeta) \, d\zeta \text{ for } i=1,\ldots,N.\label{e:IC_alpha}
\end{eqnarray}

The model for the evolution of the coefficients for arbitrary $N$ can be derived by inserting the ansatz \eqref{expansion} into the vertically resolved system \eqref{e:vertical1}-\eqref{e:vertical2} and integrating over $\zeta \in [0,1]$. According to \cite{kowalski2018moment} this leads to 
\begin{eqnarray}
    \partial_t h + \partial_x h u_m \hspace{-0.3cm}&=&\hspace{-0.3cm}0, \\
    \partial_t h u_m + \partial_x \left( h u_m^2 + h \sum_{j=1}^{N} \frac{\alpha_j^2}{2j+1} + \frac{g}{2} h^2 \right) \hspace{-0.3cm}&=&\hspace{-0.3cm}  -\frac{\nu}{\lambda} \left( u_m + \sum_{j=1}^{N} \alpha_j \right) - g h \partial_x b, \\
    \partial_t h \alpha_i + \partial_x \left( 2h u_m \alpha_i  + h \sum_{j,k=1}^{N} A_{ijk} \alpha_j \alpha_k \right) \hspace{-0.3cm}&=&\hspace{-0.3cm} u_m \partial_x h\alpha_i - \sum_{j,k=1}^{N} B_{ijk} \partial_x (h\alpha_j) \alpha_k \\
    &&\hspace{-0.3cm}- (2i+1)\frac{\nu}{\lambda} \left( u_m + \sum_{j=1}^{N} \left( 1 + \frac{\lambda}{h} C_{ij} \right) \alpha_j \right), 
\end{eqnarray}
for $i=1,\ldots,N$, the unknown variables $\left(h,u,\alpha_1, \ldots, \alpha_N \right)$ and 
\begin{equation}
    A_{ijk} = (2i+1) \int_{0}^{1} \phi_i \phi_j \phi_k \,d\zeta,
 \end{equation}
\begin{equation}
    B_{ijk} = (2i+1) \int_{0}^{1} \partial_{\zeta} \phi_i \left( \int_{0}^{\zeta} \phi_j \, d\hat{\zeta} \right) \phi_k \,d\zeta,
\end{equation}
\begin{equation}
    C_{ij} = \int_{0}^{1} \partial_{\zeta} \phi_i \partial_{\zeta} \phi_j \,d\zeta.
\end{equation}

The model can be written in closed form with the precomputed terms $A_{ijk}, B_{ijk}, C_{ij}$ for large $N$. We then write it as
\begin{equation}\label{SWME_arbitrary}
    \partial_t W + \frac{\partial F}{\partial W} \partial_x W = Q \partial_x W + P,
\end{equation}
with variables $W = \left(h, h u_m, h\alpha_1, \ldots, h\alpha_N \right)^T \in \mathbb{R}^{N+2}$, the flux Jacobian (also called conservative matrix) $\frac{\partial F}{\partial W}$ given by
\begin{equation*}
    \frac{\partial F}{\partial W}=\begin{pmatrix}
    0 & 1 & 0 & \hdots & 0 \\
    g h -u_m^2 - \displaystyle\sum_{i=1}^N \frac{\alpha_i}{2i+1} & 2u_m & \frac{2\alpha_1}{2\cdot 1 +1} & \hdots & \frac{2\alpha_N}{2N+1} \\
    -2u_m\alpha_1 - \displaystyle\sum_{j,k=1}^N A_{1jk} \alpha_j \alpha_k & 2\alpha_1 & 2 u_m \delta_{11} + 2 \displaystyle\sum_{k=1}^N A_{11k} \alpha_k & \hdots & 2 u_m \delta_{1N} + 2 \displaystyle\sum_{k=1}^N A_{1Nk} \alpha_k\\
    \vdots & \vdots & \vdots & \ddots & \vdots\\
    -2u_m\alpha_N - \displaystyle\sum_{j,k=1}^N A_{Njk} \alpha_j \alpha_k & 2\alpha_N & 2 u_m \delta_{NN} + 2 \displaystyle\sum_{k=1}^N A_{N1k} \alpha_k & \hdots & 2 u_m \delta_{NN} + 2 \displaystyle\sum_{k=1}^N A_{NNk} \alpha_k
    \end{pmatrix},
\end{equation*}
and right-hand side non-conservative matrix $Q$ 
\begin{equation*}
    Q=\begin{pmatrix}
    0 & 0 & 0 & \hdots & 0 \\
    0 & 0 & 0 & \hdots & 0 \\
    0 & 0 & u_m \delta_{11} + \displaystyle\sum_{k=1}^N B_{11k} \alpha_k & \hdots & u_m \delta_{1N} + \displaystyle\sum_{k=1}^N B_{1Nk} \alpha_k\\
    \vdots & \vdots & \vdots & \ddots & \vdots\\
    0 & 0 & u_m \delta_{N1} + \displaystyle\sum_{k=1}^N B_{N1k} \alpha_k & \hdots & u_m \delta_{NN} + \displaystyle\sum_{k=1}^N B_{NNk} \alpha_k\\
    \end{pmatrix},
\end{equation*}
with Kronecker delta $\delta_{ij}$. The friction term $P$ on the right-hand side is defined in \cite{kowalski2018moment} as first entry $P_0 = 0$ and 
\begin{equation}
P_i = -\left(2i+1\right) \frac{\nu}{\lambda} \left( u_m + \sum_{j=1}^N \left( 1 + \frac{\lambda}{h} C_{ij} \right) \alpha_j \right), i=1,\ldots,N+1.
\end{equation}
The friction term can be given explicitly as
\begin{equation}
P_i = -\frac{\nu}{\lambda} \left(2i+1\right) \left( u_m + \sum_{j=1}^N \alpha_j \right)  -\frac{\nu}{h} 4 \left(2i+1\right) \sum_{j=1}^N a_{i,j} \alpha_j, i=1,\ldots,N+1,
\end{equation}
where the constants $a_{i,j}$ are computed by
\begin{equation}
a_{i,j}=
\begin{cases}
0 \hspace{2.8cm} \quad\quad \textrm{ if } i+j =  \textrm{even},\\
\frac{\min(i-1,j) \left( \min(i-1,j) +1\right)}{2} \quad \textrm{ if } i+j =  \textrm{odd}.
\end{cases}
\end{equation}
Note that the right-hand side friction term can become quite stiff for large $N$, even though the friction coefficients $\lambda, \nu$ can be of order 1. This should be accounted for by appropriate numerical methods, e.g. Projective Integration \cite{Gear2003ProjectiveMF,Lafitte2010}. For most of this work, we will neglect the friction terms but consider non-zero topography changes $\partial_x b$.

All the models covered in this paper use the form \eqref{SWME_arbitrary} for different simplifications of the conservative and non-conservative matrix. 

We first consider the case $N=1$, also called the first order system. This model is described in \cite{kowalski2018moment} and \cite{Koellermeier2020}. The velocity profile is then given depending on the mean velocity $u_m$ and the coefficient $\alpha=\alpha_1$ as
\begin{equation}
    u(t,x,z) = = u_m(t,x) + \left( 1 - 2 \frac{z-b}{h}\right) \alpha(t,x).
\end{equation}

Note that the two values for the velocity at the top ($z=b+h$) and at the bottom ($z=b$) are given by
\begin{eqnarray}
    u(z=b+h) &=& u_m - \alpha, \\
    u(z=b) &=& u_m + \alpha. 
\end{eqnarray}
It seems reasonable, to require $u(z)$ to have the same sign over the whole velocity profile. Otherwise, the flow can no longer be approximated by means of a \emph{shallow} model assumption, as a vortex could form. Thus we require in this paper
\begin{equation}\label{e:assumption}
   |\alpha(t,x)| \leq |u(t,x)|.
\end{equation}
Compare Figure \ref{fig:switch_profiles}.

\begin{figure}[h]
		\begin{subfigure}{0.5\textwidth}
		    \centering
			\includegraphics[width=0.4\textwidth]{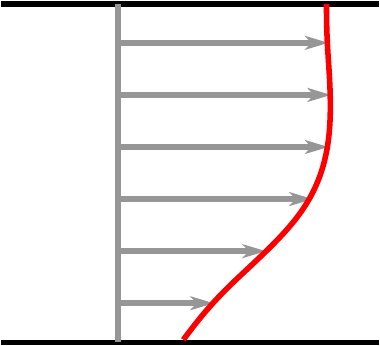}
			\caption{no change of sign}
		    \label{fig:no_change_of_sign}
		\end{subfigure}
		\begin{subfigure}{0.5\textwidth}
		    \centering
			\includegraphics[width=0.4\textwidth]{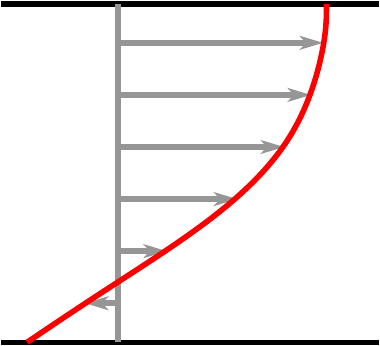}
			\caption{change of sign}
		    \label{fig:change_of_sign}
		\end{subfigure}
        \caption{Velocity profile without change of sign (a) and with change of sign (b). We only consider velocity profiles without change of sign in this paper.}
        \label{fig:switch_profiles}
\end{figure}

Choosing this linear velocity change with the vertical variable, the first order shallow water moment model reads \cite{kowalski2018moment}
\begin{equation}
    \partial_t
    \begin{pmatrix}
    h\\
    h u_m\\
    h\alpha_1\\
    \end{pmatrix} +\partial_x
    \begin{pmatrix}
    h u_m\\
    h u_m^2+\frac{1}{2}g h^2 +\frac{1}{3}h\alpha_1^2\\
    2h u_m\alpha_1
    \end{pmatrix} =Q\partial_x
    \begin{pmatrix}
    h\\
    h u_m\\
    h\alpha_1
    \end{pmatrix}-
    \begin{pmatrix}
    0\\
    gh \partial_x b\\
    0
    \end{pmatrix}
    -\frac{\nu}{\lambda}
    P,
    \label{firstordersystem}
\end{equation}
with
\begin{equation*}
    Q=\begin{pmatrix}
    0&0&0\\
    0&0&0\\
    0&0&u_m
    \end{pmatrix}, P=\begin{pmatrix}
    0\\
    u_m+\alpha_1\\
    3\left(u_m+\alpha_1+4\frac{\lambda}{h}\alpha_1\right)
    \end{pmatrix} \text{ and Jacobian }
    \frac{\partial F}{\partial V}=\begin{pmatrix}
     0 & 1 & 0 \\
    -u_m^2-\frac{\alpha_1^2}{3} +g h & 2 u_m & \frac{2\alpha_1}{3} \\
    -2 u_m \alpha_1 & 2 \alpha_1 & 2 u_m
    \end{pmatrix},
\end{equation*}
leading to the system matrix
\begin{equation}\label{AmatrixN1}
    A=\frac{\partial F}{\partial V}-Q=\begin{pmatrix}
     0 & 1 & 0 \\
    -u_m^2-\frac{\alpha_1^2}{3} +g h & 2 u_m & \frac{2\alpha_1}{3} \\
    -2 u_m \alpha_1 & 2 \alpha_1 & u_m
    \end{pmatrix}.\\
\end{equation}

The first order system has the distinct real eigenvalues
\begin{equation}
    \lambda_{1,2}=u_m\pm \sqrt{g h+\alpha_1^2}\, \textrm{ and } \lambda_3=u_m.
    \label{e:EV_SWMEN1}
\end{equation}
For positive water height $h>0$, the first order shallow water moment model is hyperbolic.\\

So far, there has been no analysis of the first order system except for the eigenvalues in \cite{Koellermeier2020,kowalski2018moment}. In this paper, we investigate the steady state of the model.

For flat bottom $\partial_x b=0$ and zero friction, the steady state fulfills
\begin{eqnarray}
    \partial_x \left(h u_m\right) &=& 0,\\
    \partial_x \left( h u_m^2+\frac{1}{2}g h^2 + \frac{1}{3}h\alpha^2 \right) &=& 0,\\
    \partial_x \left( 2h u_m\alpha \right) &=& u_m \partial_x \left( h\alpha \right),
\end{eqnarray}
From the first and last equation, we obtain after some modification
\begin{eqnarray}
    h u_m &=& const,\\
    u_m=0 ~\textrm{ or }~ \frac{\alpha}{h} &=& const.
\end{eqnarray}
Using those relations in the remaining second equation, we can derive the Rankine-Hugoniot conditions from a given state $\left( h_0, h_0 u_{m,0}, h_0\alpha_0 \right)$ to a state $\left( h, h u_m, h\alpha \right)$ and obtain (after some modifications)
\begin{equation}
    (h-h_0)\left[ -\frac{u_{m,0}^2}{g h_0} + \frac{1}{2}\left(\left(\frac{h}{h_0}\right)^2+ \left(\frac{h}{h_0}\right)\right) + \frac{1}{3}\frac{\alpha_0^2}{g h_0} \left( \left(\frac{h}{h_0}\right)^3+ \left(\frac{h}{h_0}\right)^2 + \left(\frac{h}{h_0}\right) \right) \right] = 0.
\end{equation}
We now use the following dimensionless flow numbers:
\begin{eqnarray}
    Fr = \frac{u_{m,0}}{\sqrt{g h_0}},\\
    M\alpha = \frac{\alpha_0}{u_{m,0}},
\end{eqnarray}
and write $y=\frac{h}{h_0}$ to arrive at the two solutions
\begin{equation}\label{e:RKN1}
    h=h_0 \quad \vee \quad -Fr^2 + \frac{1}{2}\left(y^2+ y\right) + \frac{1}{3} {M \hspace{-0.05cm}\alpha}^2 Fr^2 \left( y^3+ y^2 + y \right) = 0.
\end{equation}
That means that the jump conditions for the SWME with $N=1$ lead to a third order polynomial with two parameters which are the flow numbers $Fr$ and $M\alpha$, a consistent extension from the standard case of the Shallow Water Equations. The new parameter $M\alpha$ measures how far away the flow is from the standard shallow water model. For $M\alpha=0$, the shallow water equations are recovered with a constant velocity profile, whereas for $|M\alpha|=1$, the flow velocity is changing the most along the $z$-axis. For values $|M\alpha|>1$, the assumption \eqref{e:assumption} is no longer fulfilled and the model assumption of a shallow flow is not valid any more.

Note that the third order polynomial in \eqref{e:RKN1} always has at least one real zero.\\

For a smooth frictionless flow including a bottom topography, the steady state momentum equation can be modified using the mass equation to
\begin{equation}
    \partial_x \left( \frac{1}{2}u_m^2 + g(h+b) + \frac{1}{2}\alpha^2\right) = 0.
\end{equation}
The non-trivial steady state solution can thus be found using
\begin{eqnarray}
    h u_m &=& const,\\
    \frac{1}{2}u_m^2 + g(h+b)  + \frac{1}{2}\alpha^2 &=& const,\\
    \frac{\alpha}{h} &=& const.
\end{eqnarray}

In Section \ref{sec:numerics}, we will use this form of the non-trivial steady state solution to preserve steady states within the numerical scheme. 

Unfortunately, it is not possible to extend the study of steady states of the SWME form $N=1$ to $N>1$. The first problem is that the SWME loose hyperbolicity for $N>1$ as analyzed in detail in \cite{Koellermeier2020}. Hyperbolicity is a mathematical requirement for first order partial differential equations to be robust against small perturbations of the initial data, a key property of the real-world physical processes \cite{Serre1999}. 
The model is only hyperbolic for certain states depending on the values of the coefficients $\alpha_i$. As one example, consider the case $N=2$. This so-called second order moment model is given by
\begin{equation}\label{e:SWME_N2}
    \partial_t
    \begin{pmatrix}
    h\\
    h u_m\\
    h\alpha_1\\
    h\alpha_2
    \end{pmatrix} +\partial_x
    \begin{pmatrix}
    h u_m \\
    h u_m^2+g\frac{h^2}{2} +\frac{1}{3}h\alpha_1^2+\frac{1}{5}h\alpha_2^2\\
    2h u_m\alpha_1+\frac{4}{5}h\alpha_1\alpha_2\\
    2h u_m\alpha_2 + \frac{2}{3}h\alpha_1^2+\frac{2}{7}h\alpha_2^2
    \end{pmatrix} =Q\partial_x
    \begin{pmatrix}
    h\\
    h u_m\\
    h\alpha_1\\
    h\alpha_2
    \end{pmatrix}-
    \begin{pmatrix}
    0\\
    g h \partial_x b\\
    0\\
    0
    \end{pmatrix}
    -\frac{\nu}{\lambda}
    P
\end{equation}
with
\begin{equation*}
    Q=\begin{pmatrix}
    0&0&0&0\\
    0&0&0&0\\
    0&0&u_m-\frac{\alpha_2}{5}&\frac{\alpha_1}{5}\\
    0&0&\alpha_1&u_m+\frac{\alpha_2}{7}
    \end{pmatrix}
    \text{ and } P=\begin{pmatrix}
    0\\
    u_m+\alpha_1+\alpha_2\\
    3\left(u_m+\alpha_1+\alpha_2+4\frac{\lambda}{h}\alpha_1\right)\\
    5\left(u_m+\alpha_1+\alpha_2 +12\frac{\lambda}{h}\alpha_2\right)
    \end{pmatrix}.
\end{equation*}
Where the two coefficients are now $\alpha_1, \alpha_2$.

This leads to the Jacobian 
\begin{equation*}
    \frac{\partial F}{\partial V}=\left(
    \begin{array}{cccc}
     0 & 1 & 0 & 0 \\
     -\frac{\alpha_1^2}{3}-u_m^2+g h-\frac{\alpha_2 ^2}{5} & 2 u_m & \frac{2 \alpha_1}{3} & \frac{2 \alpha_2 }{5} \\
     -2 \alpha_1 u_m-\frac{4}{5}  \alpha_1 \alpha_2 & 2 \alpha_1 & 2 u_m+\frac{4 \alpha_2 }{5} & \frac{4 \alpha_1}{5} \\
      -\frac{2}{3} \alpha_1^2-2 \alpha_2 u_m - \frac{2}{7} \alpha_2^2  & 2 \alpha_2  & \frac{4 \alpha_1}{3} & 2 u_m+\frac{4 \alpha_2 }{7} \\
    \end{array}
    \right)
\end{equation*}
and the full system matrix reads 
\begin{equation}
    \label{2DSystemMatrix}
    A=\frac{\partial F}{\partial V}-Q=\begin{pmatrix}
    0&1&0&0\\
    -\frac{\alpha_1^2}{3}-u_m^2+g h-\frac{\alpha_2 ^2}{5} & 2 u_m & \frac{2 \alpha_1}{3} & \frac{2 \alpha_2 }{5} \\
    -2 \alpha_1 u_m-\frac{4}{5}  \alpha_1 \alpha_2&2\alpha_1&u_m+\alpha_2 &\frac{3\alpha_1}{5}\\
    -\frac{2}{3}\alpha_1^2-2u_m\alpha_2 -\frac{2}{7}\alpha_2^2&2\alpha_2 &-\frac{\alpha_1}{3}&u_m+\frac{3\alpha_2 }{7}
    \end{pmatrix}.
\end{equation}

However, the system is not hyperbolic and the non-hyperbolic regions are clearly shown in Figure \ref{fig:hypN2}. In particular, the eigenvalues depend on $\alpha_1$ and $\alpha_2$. It was shown in \cite{Koellermeier2020} that the non-hyperbolic regions can be reached in standard simulations which makes the SWME models with $N>1$ prone to stability problems.
\begin{figure}[htbp]
    \centering
        \includegraphics[width=0.4 \textwidth]{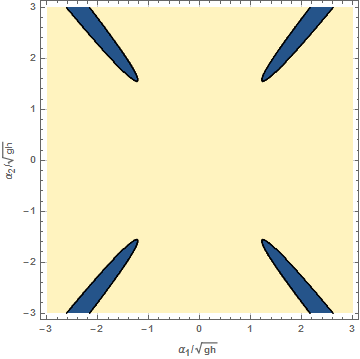}
        \caption{Non-hyperbolic region of $N=2$ model in blue, from \cite{Koellermeier2020}.}
        \label{fig:hypN2}
\end{figure}

There are several hyperbolic regularization of the SWME with arbitrary $N$ that restore hyperbolicity and yield more stable solutions while achieving similar accuracy as the original model. For more details, see \cite{Koellermeier2020}. However, it is very difficult to investigate the steady states for these models as the number of non-conservative terms is large. At the same time, those models do not depend on the higher order coefficients $\alpha_i$ any more, which leads to a drastic simplification.

\section{Shallow Water Linearized Moment Equations}
\label{sec:SWLME}
In the previous section we have seen that the general SWME lacks hyperbolicity and a proper analysis of steady states is difficult due to the non-conservative terms. Note that even the hyperbolic HSWME model and the related $\beta$-HSWME model in \cite{Koellermeier2020} pose the same problems for computing the steady states. 

In this paper, we propose a new hyperbolic model for the simulation of shallow flows, which is called Shallow Water Linearized Moment Equations (SWLME). Its derivation is based on the insights from the SWME $N=1$ model. We saw that the steady states are easy to obtain as long as there are not that many non-conservative terms in the model and as long as the higher-order equations for the variables $h \alpha_i$ are not too complicated. 

The difficult expressions in the higher-order equations are obtained by the non-linear terms $\partial_x \left( hu^2 \right)$ and $\partial_{\zeta} \left( hu \omega \right)$ in the vertically-resolved system \eqref{e:vertical1}-\eqref{e:vertical2},
which require the computation of the following terms after insertion of the ansatz \eqref{expansion}
\begin{equation*}
    \int_{0}^{1} \phi_i u^2 \, d\zeta \quad\text{ and }\quad \int_{0}^{1} \phi_i \partial_{\zeta} \left(u \omega\right) \, d\zeta.
\end{equation*}
Following an exact derivation, the first term evaluates to
\begin{eqnarray}
    \int_{0}^{1} \phi_i u^2 \, d\zeta&=& \int_{0}^{1} \phi_i \left(u_m+\sum_{j=1}^{N} \alpha_j \phi_j \right)^2 \, d\zeta \\
    &=& u_m^2 \int_{0}^{1} \phi_i \, d\zeta + \sum_{j=1}^{N} 2 u_m \alpha_j \int_{0}^{1} \phi_i \phi_j \, d\zeta + \sum_{j,k=1}^{N} 2 \alpha_j \alpha_k \int_{0}^{1} \phi_i \phi_j \phi_k \, d\zeta \\
    &=& 0 + \frac{2}{2i+1} u_m \alpha_i+ \frac{1}{2i+1} \sum_{j,k}^{N} A_{ijk} \alpha_j \alpha_k.
\end{eqnarray}
Assuming small deviations from a constant profile, i.e., $\alpha_i = \mathcal{O}\left(\epsilon\right)$ allows for neglecting the last term containing the coefficient coupling $\alpha_j \alpha_k = \mathcal{O}\left(\epsilon^2\right)$. This results in
\begin{equation*}
    \int_{0}^{1} \phi_i u^2 \, d\zeta \approx \frac{2}{2i+1} u_m \alpha_i.
\end{equation*}
The second term exactly evaluates to
\begin{eqnarray}
    \int_{0}^{1} \phi_i \partial_{\zeta} \left(u \omega\right) \, d\zeta &=& -\frac{1}{2i+1} u_m \partial_x (h\alpha_i) + \sum_{j,k}^{N} B_{ijk} \alpha_j \partial_x (h \alpha_k).
\end{eqnarray}
Again assuming small coefficients $\alpha_i = \mathcal{O}\left(\epsilon\right)$ that only change moderately, the last term containing the coefficient coupling $\alpha_j \partial_x (h \alpha_k)$ is neglected. This results in
\begin{equation*}
    \int_{0}^{1} \phi_i \partial_{\zeta} \left(u \omega\right) \, \approx -\frac{1}{2i+1} u_m \partial_x (h\alpha_i).
\end{equation*}
This leads to two changes in the equation system: 
\begin{itemize}
    \item[1.] The left-hand side transport term does no longer include the non-linear couplings between different $\alpha_i$.
    \item[2.] The right-hand side non-conservative term does no longer contain coupling terms between different $\alpha_i$.
\end{itemize}
Due to the linearization, the new model is called Shallow Water Linearized Moment Equations (SWLME).

\begin{remark}
    The linearization procedure outlined for the SWLME is related to the hyperbolic regularization procedure that leads to the so-called Hyperbolic Moment Equations (HME) for rarefied gases in \cite{Cai2013}, which are linearized around the equilibrium point in conservative variables. Another similar linearization was performed in the derivation of the so-called Simplified Hyperbolic Moment Equations (SHME) for rarefied gases in \cite{Koellermeier2016a}, which neglects the non-linearity in the ansatz to derive a hyperbolic but much simpler moment model.
\end{remark}

To see the effect of the changes in practice, consider the simple case $N=2$ that will later be extended for larger $N$. The model reads
\begin{equation*}
    \partial_t
    \begin{pmatrix}
    h\\
    h u_m\\
    h\alpha_1\\
    h\alpha_2
    \end{pmatrix} +\partial_x
    \begin{pmatrix}
    h u \\
    h u_m^2+g\frac{h^2}{2} +\frac{1}{3}h\alpha_1^2+\frac{1}{5}h\alpha_2^2\\
    \red{2h u_m\alpha_1} \\
    \red{2h u_m\alpha_2} 
    \end{pmatrix} ={\red Q}\partial_x
    \begin{pmatrix}
    h\\
    h u\\
    h\alpha_1\\
    h\alpha_2
    \end{pmatrix}
    -\frac{\nu}{\lambda}
    P
\end{equation*}
with
\begin{equation*}
    {\red Q}=\begin{pmatrix}
    0&0&0&0\\
    0&0&0&0\\
    0&0&\red{u_m} & \red{0}\\
    0&0&\red{0}&\red{u_m} 
    \end{pmatrix}
    \text{ and } P=\begin{pmatrix}
    0\\
    u_m+\alpha_1+\alpha_2\\
    3(u_m+\alpha_1+\alpha_2+4\frac{\lambda}{h}\alpha_1)\\
    5(u_m+\alpha_1+\alpha_2 +12\frac{\lambda}{h}\alpha_2)
    \end{pmatrix}.
\end{equation*}
The changed entries are given in red, illustrating the derivation above. While the model looks simpler than the SWME model \eqref{e:SWME_N2}, in comparison with the HSWME from \cite{Koellermeier2020}, the differences are smaller as the HSWME model also neglects the high-order linear terms. Most importantly, the momentum equation, which is the second equation of the model, is exactly recovered by the SWLME and the system matrix $A$ still depends on the second coefficient $\alpha_2$, which is both not the case for the HSWME model. The system matrix is given by
\begin{equation}\label{e:SWLMEN2Matrix}
    A=\begin{pmatrix}
    0&1&0&0\\
    -\frac{\alpha_1^2}{3}-u_m^2+g h-\frac{\alpha_2 ^2}{5} & 2 u_m & \frac{2 \alpha_1}{3} & \frac{2 \alpha_2 }{5} \\
    \red{-2 u_m\alpha_1} &2\alpha_1&\red{u_m} &\red{0}\\
    \red{-2 u_m\alpha_2} &2\alpha_2 &\red{0}&\red{u_m}
    \end{pmatrix}.
\end{equation}
Albeit being a simpler model, the model captures most of the original model, including the conservation of mass and momentum and the dependence of the momentum terms $hu$ on the higher order equations. The second column of the system matrix is not changed at all, leading to the correct momentum influence on the higher order equations. Only the coupling between the higher-order equations, induced by the non-linear parts (e.g. $h\alpha_1\alpha_2$ and $h\alpha_2^2$ and the additional non-conservative terms) is reduced. However, there is still a non-linear velocity and momentum coupling between all higher-order equations.

This procedure can be generalized to an explicit system for arbitrary $N$ following the same strategy. The model equations read:
\begin{equation}\label{hyperbolic_system_N}
    \partial_t
    \begin{pmatrix}
    h\\
    h u_m\\
    h\alpha_1\\
    \vdots\\
    h\alpha_N
    \end{pmatrix} +\partial_x
    \begin{pmatrix}
    h u_m \\
    h u_m^2+g\frac{h^2}{2} +\frac{1}{3}h\alpha_1^2+\ldots+\frac{1}{2N+1}h\alpha_N^2\\
    \red{2h u_m\alpha_1} \\
    \vdots \\
    \red{2h u_m\alpha_N} 
    \end{pmatrix} ={\red Q}\partial_x
    \begin{pmatrix}
    h\\
    h u_m\\
    h\alpha_1\\
    \vdots\\
    h\alpha_N
    \end{pmatrix}+
    P.
\end{equation}

The non-conservative term is simplified to
\begin{equation*}
    {\red Q} = \diag \left(0,0,{\red u_m,\hdots,u_m}\right).
\end{equation*}
The system matrix of the new SWLME then reads
\begin{equation}\label{e:A_nHWSME}
    A_N=\begin{pmatrix}
    0&1&0&\vdots&0\\
    gh -u_m^2-\frac{\alpha_1^2}{3} - \hdots -\frac{\alpha_N ^2}{2N+1} & 2 u_m & \frac{2 \alpha_1}{3} & \hdots & \frac{2 \alpha_N }{2N+1} \\
    \red{-2 u_m\alpha_1} &2\alpha_1&\red{u_m} & &\\
    \red{\vdots} &\vdots & & \red{\ddots} & \\
    \red{-2 u_m\alpha_N} &2\alpha_N & & &\red{u_m}
    \end{pmatrix} \in \mathbb{R}^{(N+2) \times (N+2)}.
\end{equation}
For the model with general $N>2$, the same observations as for the $N=2$ model hold, including the conservation of mass and momentum as well as the exact second column of the system matrix. The coupling between the higher-order equations is reduced, but still present. 

An analysis of the system matrix reveals the following theorem.
\begin{theorem}
    \label{TheoremEVnHSWME}
    The SWLME system matrix $A_N \in \mathbb{R}^{(N+2)\times(N+2)}$ \eqref{e:A_nHWSME} has the following characteristic polynomial
    \begin{equation*}
        \chi_{A_N} (\lambda) = \left(u_m-\lambda \right) \left[(\lambda-u_m)^2 - gh -  \sum_{i=1}^N \frac{3 \alpha_i^2}{2i+1} \right]
    \end{equation*}
    and the eigenvalues are given by
    \begin{equation}\label{SWLME_eigenvalues}
        \lambda_{1,2} = u_m \pm \sqrt{gh + \sum_{i=1}^N \frac{3 \alpha_i^2}{2i+1}} \quad \textrm{ and } \quad \lambda_{i+2} = u, ~\textrm{ for }~ i=1,\ldots,N.
    \end{equation}
    The system is thus hyperbolic.
\end{theorem}

\begin{proof}
    The proof closely follows the proof of the characteristic polynomial of the HSWME system matrix in \cite{Koellermeier2020}. However, we can compute the characteristic polynomial and all eigenvalues explicitly here.
    
    We write $\widetilde{\lambda}=\lambda - u_m$, so that we can compute the characteristic polynomial using
    \begin{eqnarray*}
        \chi_{A_N} (\lambda)  &=& \det \left(A_N - \lambda I \right) \\
        &=& \det \left( A_N - \left(\widetilde{\lambda}+u_m\right) I \right).
    \end{eqnarray*}

    When writing $A_N$, the following notation is used for conciseness:
    \begin{equation*}
    \begin{array}{c}
        d_0 = gh + \sum_{i=1}^N \frac{\alpha_i^2}{2i+1}, \\
        \quad d_i = -2u \alpha_i, \textrm{ for } i=1,\ldots,N\\
        \quad c_i = 2 \alpha_i, \textrm{ for } i=1,\ldots,N\\
        \quad b_i = \frac{2\alpha_i}{2i+1}, \textrm{ for } i=1,\ldots,N.
        \end{array}
    \end{equation*}
    
    Computing the determinant $\left| A_N - \left(\widetilde{\lambda}+u_m \right) I \right|$ by developing with respect to the first row yields
    \begin{eqnarray*}
    && \left| A_N - \left(\widetilde{\lambda}+u_m \right) I \right| = \left|
        \begin{array}{ccccc}
        -\widetilde{\lambda} - u_m  & 1 &  &  &  \\
        d_0 & u_m-\widetilde{\lambda} & b_1 & \dots & b_N\\
         d_1 & c_1 & -\widetilde{\lambda} &   &  \\
         \vdots & \vdots &  & \ddots &    \\
        d_N & c_N &  &  & -\widetilde{\lambda}\\
        \end{array}
        \right| \\
    &=& \left(-\widetilde{\lambda} - u_m\right) \cdot \underbrace{\left| \begin{array}{cccc}
         u_m-\widetilde{\lambda} & b_1 & \dots & b_N\\
          c_1 & -\widetilde{\lambda} &   &  \\
          \vdots &  & \ddots &    \\
         c_N &  &  & -\widetilde{\lambda}\\
        \end{array}
        \right|}_{= C_{N+1} \in \mathbb{R}^{(N+1)\times(N+1)}} - 1 \cdot \underbrace{\left| \begin{array}{cccc}
         d_0 & b_1 & \dots & b_N\\
          d_1 & -\widetilde{\lambda} &   &  \\
          \vdots &  & \ddots &    \\
         d_N &  &  & -\widetilde{\lambda}\\
        \end{array}
        \right|}_{= D_{N+1} \in \mathbb{R}^{(N+1)\times(N+1)}}\\
    \end{eqnarray*}
    
    The determinants of $C_{N+1},D_{N+1}$ are computed by developing with respect to the last row as
    \begin{equation*} \left| C_{N+1} \right|=
        \left| \begin{array}{cccc}
         u_m-\widetilde{\lambda} & b_1 & \dots & b_N\\
          c_1 & -\widetilde{\lambda} &   &  \\
          \vdots &  & \ddots &    \\
         c_N &  &  & -\widetilde{\lambda}\\
        \end{array}
        \right| = (-1)^{N+2} c_N \underbrace{\left| \begin{array}{cccc}
         b_1 & \dots & b_{N-1} & b_N\\
          -\widetilde{\lambda} &   &  &\\
           & \ddots & &   \\
          &  & -\widetilde{\lambda} &\\
        \end{array}
        \right|}_{= B_{N} \in \mathbb{R}^{N\times N}} + (-1)^{2N+2} \left(-\widetilde{\lambda}\right) \left| C_{N} \right|
    \end{equation*}
    and
    \begin{equation*} \left| D_{N+1} \right|=
        \left| \begin{array}{cccc}
         d_0 & b_1 & \dots & b_N\\
          d_1 & -\widetilde{\lambda} &   &  \\
          \vdots &  & \ddots &    \\
         d_N &  &  & -\widetilde{\lambda}\\
        \end{array}
        \right| = (-1)^{N+2} d_N \underbrace{\left| \begin{array}{cccc}
         b_1 & \dots & b_{N-1} & b_N\\
          -\widetilde{\lambda} &   &  &\\
           & \ddots & &   \\
          &  & -\widetilde{\lambda} &\\
        \end{array}
        \right|}_{= B_{N} \in \mathbb{R}^{N\times N}} + (-1)^{2N+2} \left(-\widetilde{\lambda}\right) \left| D_{N} \right|.
    \end{equation*}
    
    The determinant of $B_{N}$ is easily computed as 
    \begin{equation*}
        \left| B_{N} \right|= (-1)^{N+1} b_N \left(-\widetilde{\lambda}\right)^{N-1}.
    \end{equation*}
    With the help of this, we get
    \begin{equation*}
        \left| C_{N+1} \right|= -c_N b_N  \left(-\widetilde{\lambda}\right)^{N-1} + \left(-\widetilde{\lambda}\right) \left| C_{N} \right|= \ldots = \left(-\widetilde{\lambda}\right)^{N-1} \left( - \sum_{i=1}^N c_i b_i\right) + \left(-\widetilde{\lambda}\right)^{N} \underbrace{\left( u_m -\widetilde{\lambda}\right)}_{= \left| C_{1} \right|}
    \end{equation*}
    and analogously 
    \begin{equation*}
        \left| D_{N+1} \right|= -d_N b_N  \left(-\widetilde{\lambda}\right)^{N-1} + \left(-\widetilde{\lambda}\right) \left| D_{N} \right|= \ldots = \left(-\widetilde{\lambda}\right)^{N-1} \left( - \sum_{i=1}^N d_i b_i\right) + \left(-\widetilde{\lambda}\right)^{N} \underbrace{d_0}_{= \left| D_{1} \right|}.
    \end{equation*}
    Note that $- \sum_{i=1}^N c_i b_i = - \sum_{i=1}^N \frac{4 \alpha_i^2}{2i+1}$, $- \sum_{i=1}^N d_i b_i = \sum_{i=1}^N \frac{4 \alpha_i^2}{2i+1} u_m$, $d_0 = gh + \sum_{i=1}^N \frac{\alpha_i^2}{2i+1}$.
    
    Next, insertion of these terms into the characteristic polynomial of the system matrix $A_{N}$ yields
    \begin{eqnarray*}
        \left| A_N - \left(\widetilde{\lambda}+u_m \right) I \right| &=& \left(-\widetilde{\lambda}-u_m\right)\cdot\left| C_{N+1}\right| - 1\cdot \left| D_{N+1}\right|\\
        &=& \left(-\widetilde{\lambda}-u_m\right)\cdot\left[ \left(-\widetilde{\lambda}\right)^{N-1} \left( - \sum_{i=1}^N c_i b_i\right) + \left(-\widetilde{\lambda}\right)^{N} \left( u_m -\widetilde{\lambda}\right) \right] \\
        && - 1\cdot \left[ \left(-\widetilde{\lambda}\right)^{N-1} \left( - \sum_{i=1}^N d_i b_i\right) + \left(-\widetilde{\lambda}\right)^{N} d_0\right] \\
        &=& \left(-\widetilde{\lambda}\right)^{N} \left[ \widetilde{\lambda}^2 - gh - \sum_{i=1}^N \frac{3 \alpha_i^2}{2i+1} \right]\\
        &=& \left(u_m-\lambda\right)^{N} \left[ (\lambda-u_m)^2 - gh - \sum_{i=1}^N \frac{3 \alpha_i^2}{2i+1} \right],
    \end{eqnarray*}
    which proves the first part of the theorem.
    
    Setting the characteristic polynomial to zero results in the following propagation speeds of the system:
    \begin{equation*}
        \lambda_{1,2} = u_m \pm \sqrt{gh + \sum_{i=1}^N \frac{3 \alpha_i^2}{2i+1}}, \quad \textrm{ and } \quad \lambda_{i+2} = u_m, ~\textrm{ for }~ i=1,\ldots,N.
    \end{equation*}
    
    The propagation speeds prove that the system is hyperbolic for positive water height.
\end{proof}

From the form of the eigenvalues, the new model for $N\geq2$ can be seen as a consistent extension of the hyperbolic $N=1$ model from Section \ref{sec:SWME}, compare also the eigenvalues in equation \eqref{e:EV_SWMEN1}.

We remark that such an analysis is not possible for the original SWME model for arbitrary $N$ as the eigenvalues have a very complicated structure and cannot be given in analytical form. For the new hyperbolic model, the eigenvalues $\lambda_{1,2}$ still depend on all flow variables. However, the analysis can be carried out analytically. For the hyperbolic HSWME and $\beta$HSWME models in \cite{Koellermeier2020}, the eigenvalues depend solely on $\alpha_1$, which is a drastic simplification. For those models, the eigenvector analysis is still very involved and theoretical results are only possible for small values of $\left(M\alpha\right)_1 \ll 1$. In this case, the model has the same wave properties as the SWLME system. 
From a straightforward computation, the eigenvectors $v_i$ for $i=1, \ldots, N+2$ of the SWLME system can be derived as
\begin{equation}
    v_{1,2} = \begin{pmatrix}
    \frac{1}{2 \alpha_n} \\
    \displaystyle\frac{1}{2 \alpha_n} \left( u_m + \sqrt{gh \pm \sum_{i=1}^N \frac{3 \alpha_i^2}{2i+1}}\right) \\
    \frac{\alpha_1}{\alpha_N} \\
    \vdots \\
    \frac{\alpha_N}{\alpha_N}
    \end{pmatrix}
\end{equation}
\begin{equation}
    v_{i+2} = \begin{pmatrix}
    \displaystyle{\frac{6\alpha_{n+1-1}}{(2(n+1-i)+1) -3gh + \sum_{i=1}^N \frac{3 \alpha_i^2}{2i+1}}} \\
    \displaystyle\frac{6\alpha_{n+1-1}u}{(2(n+1-i)+1) -3gh + \sum_{i=1}^N \frac{3 \alpha_i^2}{2i+1}} \\
    \delta_{n+3-i,3} \\
    \vdots \\
    \delta_{n+3-i,N} 
    \end{pmatrix}, \text{ for } i=1,\ldots, N,
\end{equation}
for Kronecker delta $\delta_{i,j}$.

It can be checked that the first two eigenvalues $\lambda_{1,2}$ are genuinely non-linear, while all other eigenvalues $\lambda_{i+2}$ for $i=1,\ldots, N$ are linearly degenerate. Note that the analysis of eigenvalues and eigenvectors is not possible for the SWME system, due to the lack of hyperbolicity. The linearization within the $N$ higher moment equations during the derivation procedure consistently leads to the resulting $N$ linearly degenerate eigenvalues. However, the first two eigenvalues, corresponding to the unchanged conservation of mass and momentum, remain genuinely non-linear. The full characterization of the eigenstructure of the SWLME allows for the use of efficient numerical methods, for example using the relation between Riemann solvers and PVM methods \cite{castro2012class}.
 
Rankine-Hugoniot conditions can be derived analogously to the SWME case with $N=1$ as follows.
For flat bottom $\partial_x b=0$ and zero friction, the steady state fulfills
\begin{eqnarray}
    \partial_x \left(h u_m\right) &=& 0\\
    \partial_x \left( h u_m^2+\frac{1}{2}g h^2 + \frac{1}{3}h\alpha_1^2 + \ldots + \frac{1}{2N+1}h\alpha_N^2 \right) &=& 0\\
    \partial_x \left( 2h u_m\alpha_1 \right) &=& u_m \partial_x \left( h\alpha_1 \right)\\
     &\vdots& \\
    \partial_x \left( 2h u_m\alpha_N \right) &=& u_m \partial_x \left( h\alpha_N \right)
\end{eqnarray}
First looking at all equations except the second, we obtain after some modification
\begin{eqnarray}
    h u_m &=& const,\\
    u_m=0 \textrm{ or } \frac{\alpha_i}{h} &=& const, \textrm{ for } i=1,\ldots,N.
\end{eqnarray}
Using those relations in the remaining second equation, we can derive the Rankine-Hugoniot conditions from a given state $\left( h_0, h_0 u_{m,0}, h_0\alpha_{1,0}, \ldots, h_0 \alpha_{N,0} \right)$ to a state $\left( h, h u_m, h\alpha_{1}, \ldots, h \alpha_{N} \right)$ and obtain (after some modifications)
\begin{equation}
    (h-h_0)\left[ -\frac{u_{m,0}^2}{g h_0} + \frac{1}{2}\left(\left(\frac{h}{h_0}\right)^2+ \left(\frac{h}{h_0}\right)\right) + \sum_{i=1}^N \frac{1}{2i+1}\frac{\alpha_{i,0}^2}{g h_0} \left( \left(\frac{h}{h_0}\right)^3 + \left(\frac{h}{h_0}\right)^2 + \left(\frac{h}{h_0}\right) \right) \right] = 0.
\end{equation}
We extend the previous dimensionless flow numbers by using one number for each variable:
\begin{eqnarray}
    Fr &=& \frac{u_{m,0}}{\sqrt{gh_0}},\\
    \left(M\alpha\right)_i &=& \frac{\alpha_{i,0}}{u_{m,0}}, \quad \textrm{ for } i=1,\ldots,N,
\end{eqnarray}
writing $y=\frac{h}{h_0}$, we arrive at the two solutions
\begin{equation}
    h=h_0 ~\vee~ -Fr^2 + \frac{1}{2}\left(y^2+ y\right) + \sum_{i=1}^N \frac{1}{2i+1} \left(M\alpha\right)_i^2 Fr^2 \left( y^3+ y^2 + y \right) = 0.
\end{equation}
From the previous equation, we see a new dimensionless number ${M\alpha}^2 := \sum_{i=1}^N \frac{1}{2i+1} \left(M\alpha\right)_i^2$ appearing. The new number ${M\alpha}$ measures the total deviation from equilibrium. This leads to a consistent extension of the SWME $N=1$ test case above. We see that the Rankine-Hugoniot conditions allow for similar solutions as before, this time with $Fr$ and $M\alpha$ as dimensionless flow numbers. We note that the equations always have at least one solution for non-zero $Fr$ and $M\alpha$.

Analogously, we extend the conditions for smooth and frictionless steady states including a bottom topography. We will later use this to derive a well-balancing scheme. We can derive 
\begin{equation}
    \partial_x \left( \frac{1}{2}u_m^2 + g(h+b) + \frac{3}{2}\sum_{i=1}^N \frac{1}{2i+1}\alpha_i^2 \right) = 0.
\end{equation}
The non-trivial steady state solution can thus be found using
\begin{eqnarray}\label{eq:stationary_solutions_1}
    h u_m &=& const,\\
    \label{eq:stationary_solutions_2}
    \frac{1}{2}u_m^2 + g(h+b)  + \frac{3}{2}\sum_{i=1}^N \frac{1}{2i+1}\alpha_i^2 &=& const,\\
    \label{eq:stationary_solutions_3}
    \frac{\alpha_i}{h} &=& const, \textrm{ for } i=1,\ldots,N.
\end{eqnarray}
This expression can be used in the following numerical methods section to obtain a proper well-balancing scheme for the new model. First, we will rewrite the model in the proper form with a conservative and non-conservative part to use it in the numerical schemes thereafter.

The system \eqref{hyperbolic_system_N} with topography but without friction terms is therefore written in the form 
\begin{equation}
    U_{t} + \partial_x F(U) + B(U) \partial_x U  = S(U) \partial_x b .
\end{equation}

By straightforward calculation, we obtain
\begin{equation}
    U = \left(\begin{array}{c}
     h  \\
     hu_m \\
     h\alpha_1 \\
     \vdots \\
     h\alpha_N
\end{array}\right), \ \ F(U) = \begin{pmatrix}
    h u_m \\
    h u_m^2+g\frac{h^2}{2} +\frac{1}{3}h\alpha_1^2+\ldots+\frac{1}{2N+1}h\alpha_N^2\\
    2h u_m\alpha_1 \\
    \vdots \\
    2h u_m\alpha_N
    \end{pmatrix},
\end{equation}
\begin{equation}
    B(U) = diag(0,0,-u_m,...,-u_m).
\end{equation}

We can also write this system in the form
\begin{equation}\label{SWMEN1nonconservativeform}
    \partial_t W + \mathcal{A}(W) \partial_x W = 0,
\end{equation}
with
$$
W=\left(\begin{array}{c}
     h  \\
     h u_m \\
     h\alpha_{1} \\
     \vdots \\
     h\alpha_{N}\\
     b
\end{array}\right), \ \ \mathcal{A}(W) = \left(\begin{array}{cc}
     A(W) & -S(W) \\
     0 & 0
\end{array}\right),
$$
where $A(W)$ has the form \eqref{e:A_nHWSME} and 
$S(W) = \left(\begin{array}{c}
     0  \\
     -gh \\
     0 \\
     \vdots \\
     0
\end{array}\right)$.

For comparison we note that also the existing HSWME and $\beta$-HSWME models from \cite{Koellermeier2020} can be written in the same form, see the appendix \ref{app}.

\section{Numerical methods}
\label{sec:numerics}
In this section, we recall the general high-order well-balanced method from \cite{castro2020well} and construct the first order as well as the second order scheme for applications of the SWLME derived in the previous section. At the end of the section we will outline the specific spatial discretization scheme used for the numerical tests in the next section.

\subsection{A general high-order well-balanced procedure}
The previously derived shallow water models can all be written as non-conservative systems of the form
\begin{equation}\label{FBSsystem}
    \partial_t U + \partial_x F(U) + B(U) \partial_x U = S(U) \partial_x b.
\end{equation}
It is well known that these systems are equivalent to
\begin{equation}\label{nonconservativesystem}
    \partial_t W + \mathcal{A}(W) \partial_x W = 0,
\end{equation}
where 
$$W=\left(\begin{array}{c}
     U  \\
     b 
\end{array}\right), \ \ \mathcal{A}(W)=\left(\begin{array}{cc}
     \frac{\partial F}{\partial U}(U) + B(U) & -S(U)  \\
     0 & 0 
\end{array}\right).$$
The goal of this section is to develop a family of numerical methods that are well-balanced for the frictionless SWLME introduced before, i.e., that preserve the stationary solutions verifying (\ref{eq:stationary_solutions_1}), (\ref{eq:stationary_solutions_2}) and (\ref{eq:stationary_solutions_3}). In this section we will follow \cite{castro2020well} adding the non-conservative products. The interested reader is referred to this reference for details and proofs.\\
We consider semi-discrete finite-volume methods of the form
\be\label{eq:semi_implicit}
\frac{dW_{i}}{dt}= -\frac{1}{\Delta x}
\Big(
D^{-}_{i+\frac{1}{2}}+D^{+}_{i-\frac{1}{2}} + \int_{x_{i-\frac{1}{2}}}^{x_{i+\frac{1}{2}}} \mathcal{A}(\mathbb{P}_{i}(x))\frac{\partial}{\partial x}\mathbb{P}_{i}(x)dx
\Big),
\ee
where
\begin{itemize}
	\item $W_{i}(t) \cong \displaystyle \int_{x_{i+\frac{1}{2}}}^{x_{i-\frac{1}{2}}} W(t,x) \,dx$ is the respective cell average value, 
	\item $\mathbb{P}_{i}(x)$ is a high-order well-balanced operator in the sense defined in \cite{castro2020well}.
	\item $D_{i+\frac{1}{2}}^{\pm} = \mathbb{D}^{\pm}\left(W_{i+\frac{1}{2}}^{-}, W_{i+\frac{1}{2}}^{+}\right)$, is the respective fluctuation with reconstructed states
	$$
	W_{i+\frac{1}{2}}^{-} = \mathbb{P}_{i}(x_{i+\frac{1}{2}}),\ \ W_{i+\frac{1}{2}}^{+} = \mathbb{P}_{i+1}(x_{i+\frac{1}{2}}),
	$$
	and $\mathbb{D}(W_{l}, W_{r})$ verifies:
	\begin{equation}\label{eq:D}
	    \mathbb{D}^{-}(W_{l}, W_{r}) + \mathbb{D}^{+}(W_{l}, W_{r}) = \int_{0}^{1} \mathcal{A}(\Psi)\frac{\partial \Psi}{\partial s} \,ds,
	\end{equation}
	where $\Psi$ is a family of paths joining $W_{l}$ with $W_{r}$.
\end{itemize}

In order to design the high-order well-balanced operator we follow the strategy introduced in \cite{castro2008well}. The following steps need to be performed in order to compute $\mathbb{P}_{i}$ at the cell $[x_{i-\frac{1}{2}}, x_{i+\frac{1}{2}}]$ for a given family of cell values $\{W_{i}\}$:
\begin{enumerate}
	\item Obtaining the steady solution $W_{i}^{*}(x)$ such that:
	\be\label{eq:step1}
	\frac{1}{\Delta x}\int_{x_{i-\frac{1}{2}}}^{x_{i-\frac{1}{2}}} W_{i}^{*}(x)dx = W_{i},
	\ee
	if possible. In other cases consider $W_{i}^{*}\equiv W_{i}^{n}$.
	\item Computing the fluctuations $\{V_{j}\}_{j\in S_{i}}$ within the stencil $S_{i}$:
	\be\label{eq:step2}
	V_{j} = W_{j} - \frac{1}{\Delta r} \int_{x_{j-\frac{1}{2}}}^{x_{j+\frac{1}{2}}}W_{i}^{*}(x)dx, \ \ j\in S_{i}.
	\ee
	\item Applying the reconstruction operator with the necessary order to the fluctuations $\{V_{j}\}_{j\in S_{i}}$:
	$$
	Q_{i}(x) = Q_{i}(x; \{V_{j}\}_{j\in S_{i}}).
	$$
	\item Defining the well-balanced operator:
	$$
	\mathbb{P}_{i}(x) = W_{i}^{*}(x) + Q_{i}(x).
	$$
\end{enumerate}

$\mathbb{P}_{i}$ is well-balanced for every steady solution provided that the reconstruction operator $Q_{i}$ is exact for the null function. Moreover, it is conservative, i.e., 
$$
\frac{1}{\Delta x} \int_{x_{i-\frac{1}{2}}}^{x_{i+\frac{1}{2}}} \mathbb{P}_{i}(x) dr = W_{i}, \ \ \text{ for all }  i,
$$
provided that $Q_{i}$ is conservative, and it is high-order accurate provided that the steady solutions are smooth (see \cite{castro2020well} for details).

\subsection{First order well-balanced scheme}
We apply the steps of the previous subsection to the system (\ref{SWMEN1nonconservativeform}) in a first order setup before considering the second order scheme in the next section. As the bottom topography $b$ is known, we will focus on the other variables of the system.

The cell averages of the initial condition will be computed using the mid-point rule, that is
$$
W_{i}^{0} =  W_{0}(x_{i}), \ \ \text{ for all }  i,
$$
where $W_{0}(x)$ is the initial condition. \\
In the case of the SWLME system, the steady state solutions verify:
	\begin{eqnarray*}
	h u_m &=& C_{1} \equiv const,\\
    \frac{1}{2}u_m^2 + g(h+b)  + \frac{3}{2}\sum_{i=1}^N \frac{1}{2i+1}\alpha_i^2 &=& C_{2} \equiv const,\\
    \frac{\alpha_{1}}{h} &=& C_{3} \equiv const,\\
    \frac{\alpha_{2}}{h} &=& C_{4} \equiv const, \\
    &\vdots& \\
    \frac{\alpha_{N}}{h} &=& C_{N+2} \equiv const.
	\end{eqnarray*}
	Using the mid-point rule in (\ref{eq:step1}) the first step is to obtain, if possible, the stationary solution $W_{i}^{*}$ such that:
	\be\label{eq:step1SWMEN1}
	W_{i}^{*}(x_{i}) = W_{i}.
	\ee
	With this information the constants $C_{1}, C_{2}$, $ C_{3}$,...,$C_{N+2}$ can be computed as
	\begin{equation}\label{eq:constantN}
	\left\{\begin{array}{l}
	     C_1 = h_{i}u_{m,i}, \\
	     C_2 = \frac{1}{2}u_{m,i}^2 + g(h_{i}+b(x_{i}))  + \frac{3}{2}\sum_{j=1}^N \frac{1}{2j+1}\alpha_{j,i}^2, \\
	     C_{3} = \frac{\alpha_{1,i}}{h_{i}},\\
	     C_{4} = \frac{\alpha_{2,i}}{h_{i}},\\
	     \ \ \ \vdots \\
	     C_{N+2} = \frac{\alpha_{N,i}}{h_{i}}.
	\end{array}\right.
	\end{equation}
	Using the relations \eqref{eq:constantN}, the stationary solution can be evaluated in a point $x=a$. The evaluation of the steady state solution requires finding roots of the function
	\begin{equation}\label{eq:functionN}
	    f(h) = Dh^{4} + 2h^{3}g+2h^{2}(gb(a)-C_2)+C_{1}^{2},
	\end{equation}
	where the parameter $D$ is given by
	$$D=C_{3}^{2}+\frac{3}{5}C_{4}^{2} + \dots + \frac{3}{2N+1}C_{N+2}^{2}.$$
	The derivative of the function $f$ is given by
	$$f'(h) = 4Dh^{3} + 6h^{2}g+4h(gb(a)-C_2).$$
	The positive root $h_c$ of $f'(h)$ is
	\begin{equation}\label{eq:criticalpointN}
	    h_{c} = \frac{-3g + \sqrt{9g^{2}-16D(b(a)g-C_{2})}}{4D},
	\end{equation}
	and we can see that it is a minimum of the function $f$. An example of a function $f$ is plotted in Figure \ref{fig:Function_f_example}.
	\begin{figure}[htbp]
    \centering
        \begin{overpic}[width=0.6\textwidth]{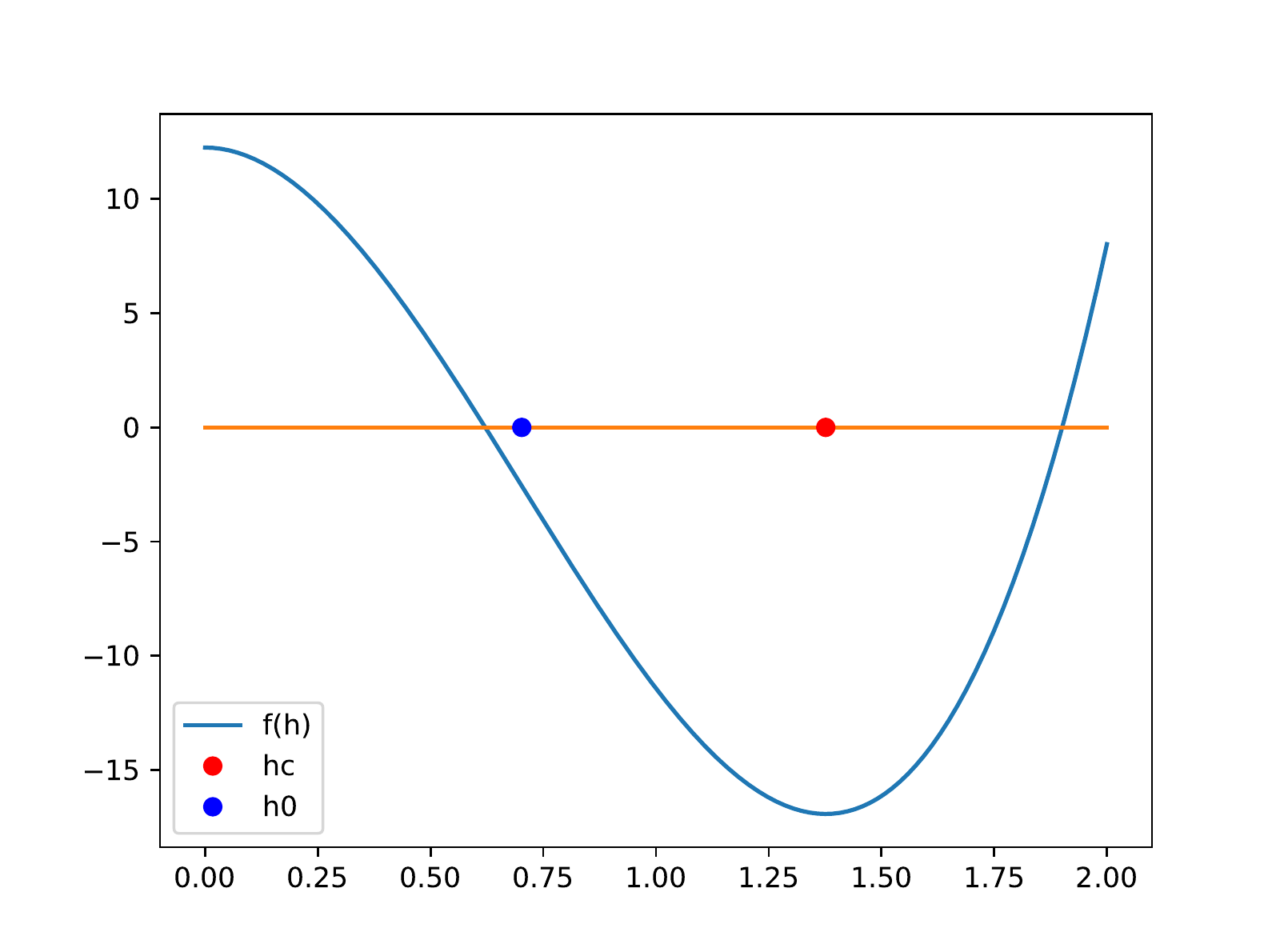}
            \put(91,4){$h$}
            \put(3, 62){$f(h)$}
        \end{overpic}
        \caption{An example of the root finding function $f(h)$ \eqref{eq:functionN} with some constants $C_{i}$. The minimum $h_c$ and the initial value of the Newton algorithm $h_0$ are shown.}
        \label{fig:Function_f_example}
    \end{figure}
	
	\begin{observacion}
	In case of $D=0$, the minimum simplifies to
	$$h_c = -\frac{2(b(a)g-C_2)}{3g}.$$
	\end{observacion}
	\begin{observacion}
	In order to find the roots of $f(h)$ the \textit{Newton-Raphson} method is employed with initial value $h_0$ that is the positive root of 
	$$f''(h) = 12Dh^{2} + 12hg+4(gb(a)-C_2),$$
	given by
	$$h_0 = \frac{-3g + \sqrt{9g^{2}-12D(b(a)g-C_{2})}}{6D}.$$
	It is easy to see that $0\leq h_0 \leq h_c$, compare also Figure \ref{fig:Function_f_example}.
	\end{observacion}
	We can conclude then the following: If $f(h_{c})<0$ there exist two possible states for $W_{i}^{*}(x_{i\pm\frac{1}{2}})$, one subcritical and one supercritical. The following criterion will be used to choose one state:
	\begin{enumerate}
	    \item If $W_i$ is subcritical or supercritical, then we will choose the solution in the same regime (subcritical or supercritical) as $W_i$ for $W_{i}^{*}(x_{i\pm\frac{1}{2}})$.
	    \item If $W_i$ is transcritical, then the solution that has the same behaviour (subcritical or supercritical) as $W_{i-1}$ will be selected for $W_{i}^{*}(x_{i-\frac{1}{2}})$ and the solution whose behaviour is the same as $W_{i+1}$ will be selected for $W_{i}^{*}(x_{i+\frac{1}{2}})$.
	\end{enumerate}
	Following the procedure described in \cite{castro2020well}, the reconstruction operator reduces to $\mathbb{P}_i(x)=W_{i}^*(x)$ and the first order numerical scheme reduces to:
\be
W_{i}^{n+1}=W_{i}^{n} -\frac{\Delta t}{\Delta r}(
D^{-}_{i+\frac{1}{2}}+D^{+}_{i-\frac{1}{2}}),
\ee
for $W_{i-\frac{1}{2}}^{+} = \mathbb{P}_{i}(x_{i-\frac{1}{2}})$ and $W_{i+\frac{1}{2}}^{-} = \mathbb{P}_{i}(x_{i+\frac{1}{2}})$,
where we have used that $\mathbb{P}(x) = W_{i}^{*}(x)$ is a steady solution.\\
In the case we could not find such a stationary solution verifying (\ref{eq:step1SWMEN1}) the standard trivial reconstruction is considered.

\subsection{Second order well-balanced scheme}
Now we consider the second order scheme for which a second order spatial reconstruction using the minmod limiter will be employed, see \cite{castro2020well}.

The cell averages of the initial condition are again computed using the mid-point rule:
$$
W_{i}^{0} =  W_{0}(x_{i}), \ \ \text{ for all }  i,
$$
where $W_{0}(x)$ is the initial condition.
\begin{enumerate}
	\item Obtaining the steady solution: 
    In the same fashion as for the first order scheme, if possible, the steady state $W_{i}^{*}$ needs to be found such that
	\be
	W_{i}^{*}(x_{i}) = W_{i}.
	\ee
	After computing the constants $C_{1}, C_{2}$, $ C_{3}$,...,$C_{N+2}$ as in (\ref{eq:constantN}),
	the stationary solution can be evaluated in a point $x=a$. In order to do this, the roots of the function $f$ in (\ref{eq:functionN}) needs to be computed. As defined in (\ref{eq:criticalpointN}), $f$ has a minimum in $h_{c}$. Again if $f(h_c)<0$ there exist two possible values for $W_{i}^{*}(i\pm 1, i\pm \frac{1}{2})$ and we use the same criterion as for the first order scheme in order to choose one.
    \item Computing the fluctuations: After the evaluation of the stationary solution in a point $r=a$ the fluctuations $\{V_{i-1},V_{i},V_{i+1}\}$ in (\ref{eq:step2}) are computed using the mid-point rule
    $$
    \begin{array}{ll}
        V_{i-1} &= W_{i-1} - W_{i}^{*}(x_{i-1}),\\
        V_{i} &= W_{i} - W_{i}^{*}(x_i) = 0,\\
        V_{i+1} &= W_{i+1} - W_{i}^{*}(x_{i+1}).
    \end{array}
    $$
    \item Applying the reconstruction operator:
    After the fluctuations are computed the $minmod$ reconstruction is used to obtain the reconstruction operator (see \cite{van1973towards})
    $$Q_{i}(x) = V_{i}+minmod\left(\displaystyle \frac{V_i-V_{i-1}}{\Delta x}, \frac{V_{i+1}-V_{i-1}}{2\Delta x}, \frac{V_{i+1}-V_{i}}{\Delta x}\right)(x-x_{i}),$$
    where
    $$
    minmod(a,b,c) = \begin{cases}
    min\{a,b,c\} & \text{if} \ \ a,b,c>0, \\
	max\{a,b,c\} & \text{if} \ \ a,b,c <0, \\
    0 & \text{otherwise.}
    \end{cases}
    $$
    \item Defining the well-balanced operator: The well-balanced reconstruction operator is given by
    $$
    \mathbb{P}_{i}(x) = W_{i}^{*}(x) + Q_{i}(x).
    $$
    \end{enumerate}

The well-balanced property can be lost if a quadrature formula is used directly in the right part of (\ref{eq:semi_implicit}), as the quadrature formula is in general not exact. Therefore, the semi-discrete scheme is first rewritten as proposed in \cite{castro2020well} taking into account the non-conservative part
	\begin{eqnarray*}
	    \frac{dW_{i}}{dt}= -\frac{1}{\Delta x}
\Big(
D^{-}_{i+\frac{1}{2}}+D^{+}_{i-\frac{1}{2}} +
\int_{x_{i-\frac{1}{2}}}^{x_{i+\frac{1}{2}}} \left(\mathcal{A}(\mathbb{P}_{i}(x))\frac{\partial}{\partial x}\mathbb{P}_{i}(x) - \mathcal{A}(W_{i}^{*}(x))\frac{\partial}{\partial x}W_{i}^{*}(x)\right)dx\\  + \int_{x_{i-\frac{1}{2}}}^{x_{i+\frac{1}{2}}} \mathcal{A}(W_{i}^{*}(x))\frac{\partial}{\partial x}W_{i}^{*}(x)dx\Big).
	\end{eqnarray*}
Once this equivalent form is obtained, we use that $W_{i}^{*}$ is a stationary solution in the second integral and then employ the mid-point rule for the first integral without losing the well-balanced property what leads to
\begin{equation}
    \frac{dW_{i}}{dt}= -\frac{1}{\Delta x}
    \Big(D^{-}_{i+\frac{1}{2}}+D^{+}_{i-\frac{1}{2}} + 	\mathcal{A}(\mathbb{P}_{i}(x_{i}))minmod\left(\displaystyle \frac{V_i-V_{i-1}}{\Delta x}, \frac{V_{i+1}-V_{i-1}}{2\Delta x}, \frac{V_{i+1}-V_{i}}{\Delta x}\right)\Big),
\end{equation}
for $W_{i-\frac{1}{2}}^{+} = \mathbb{P}_{i}(x_{i-\frac{1}{2}})$ and $W_{i+\frac{1}{2}}^{-} = \mathbb{P}_{i}(x_{i+\frac{1}{2}})$. The discretization in time is performed with a Runge-Kutta TVD method of order 2, see \cite{gottlieb1998total}.

\begin{observacion}
     The extension to higher-order is straightforward: Although not implemented in the present paper, a third order well-balanced scheme will be based on the two point Gaussian quadrature formula for computing the averages. In the first step, we need to find the constants $C_{j
    }$, $j=1,...,N+2$ such that
    $$\frac{1}{2}W_{i}^*(x_a,C_{1},...,C_{N+2}
    ) + \frac{1}{2}W_{i}^*(x_b,C_{1},...,C_{N+2}
    ) = W_i,$$
    where $x_a$ and $x_b$ are the two quadrature points and $W_{i}^*(x,C_{1},...,C_{N+2}
    )$ represents the stationary solution given by the constants $C_j$ evaluated in $x$. Then we follow the steps considering a third order reconstruction operator (e.g. CWENO reconstruction) and using again the two point Gaussian quadrature.

\end{observacion}

\subsection{Spatial discretization}
In order to completely define the scheme, what remains is to define the form of the fluctuations $D_{i+\frac{1}{2}}^{\pm}$ and the non-conservative terms in \eqref{eq:semi_implicit}, for which we use a path-consistent scheme based on segments in the conservative variables as family of paths joining two states:
$$\Psi(s;W_l,W_r) = \left(\begin{array}{c}
     \Psi_U(s;W_l,W_r)  \\
     \Psi_b(s;W_l,W_r)
\end{array}\right) = \left(\begin{array}{c}
     U_l + s(U_r-U_l) \\
     b_l + s(b_r-b_l)
\end{array}\right), \ \ s\in[0,1],$$
and a PVM-like method \cite{castro2012class} corresponding to a choice in (\ref{eq:D}) of 

\begin{eqnarray}\label{eq:PVM_method}
    D_{i+\frac{1}{2}}^{\pm} = \frac{1}{2}\left(F(U_{r})-F(U_l) + B_{i+\frac{1}{2}}(U_r-U_l) - S_{i+\frac{1}{2}}(b_r-b_l) \right. \\
    \pm \left. Q_{i+\frac{1}{2}}(U_r-U_l-\mathcal{A}_{i+\frac{1}{2}}^{-1}S_{i+\frac{1}{2}}(b_r-b_l))\right) \nonumber,
\end{eqnarray}
	
where 
$$\mathcal{A}_{i+\frac{1}{2}}= \left(\begin{array}{cc}
   A_{i+\frac{1}{2}}  & -S_{i+\frac{1}{2}} \\
    0 & 0
\end{array}\right)$$ is a generalized Roe matrix \cite{toumi1992weak} so that $A_{i+\frac{1}{2}} = J_{i+\frac{1}{2}} + B_{i+\frac{1}{2}}$ and $S_{i+\frac{1}{2}}$ have to verify
\begin{equation}\label{A_Roe}
    A_{i+\frac{1}{2}} = \int_{0}^{1}A(U_l+s(U_r-U_l))\, ds,
\end{equation}

\begin{equation}\label{J_Roe}
    J_{i+\frac{1}{2}}(U_r-U_l) = F(U_r) - F(U_l),
\end{equation}

\begin{equation}\label{B_Roe}
    B_{i+\frac{1}{2}} = \int_{0}^{1}B(U_l+s(U_r-U_l))\, ds,
\end{equation}

\begin{equation}\label{S_Roe}
    S_{i+\frac{1}{2}} = \int_{0}^{1}S(U_l+s(U_r-U_l))\, ds,
\end{equation}

and the polynomial viscosity matrix is $Q_{i+\frac{1}{2}} = P(A_{i+\frac{1}{2}})$, for polynomial $P$. The source term evaluates to
$$S_{i+\frac{1}{2}} = \left(\begin{array}{c}
     0  \\
     -g\frac{h_l+h_{r}}{2} \\
     0 \\
     \vdots \\
     0
\end{array}\right).$$

In the case of the model SWLME, it can be shown that the system (\ref{J_Roe}) leads to an evaluation of the Jacobian
\begin{equation}\label{J_Roe_r}
    J_{i+\frac{1}{2}} = \frac{\partial F}{\partial U}(h_{R}, u_{m,R}, \alpha_{1,R}, ..., \alpha_{N,R}),
\end{equation}
at the intermediate values 
$$h_R = \frac{h_l+h_{r}}{2}, \quad u_{m,R} = \frac{\sqrt{h_l}u_{m,l} + \sqrt{h_{r}}u_{m,r}}{\sqrt{h_{l}}+\sqrt{h_{r}}},$$
and
$$\alpha_{j,R} = \frac{\sqrt{h_l}h_{r}\alpha_{j,r} + \sqrt{h_{r}}h_{l}\alpha_{j,l}}{\sqrt{h_{l}}h_{r}+\sqrt{h_{r}}h_{l}}, \quad j=\{1,...,N\}.$$
\begin{observacion}
    We point out that (\ref{J_Roe_r}) is a generalization of the mean values that are obtained with the Roe matrix for the usual Shallow Water equations.
\end{observacion}
From (\ref{B_Roe}) we obtain that $B_{i+\frac{1}{2}}$ is an evaluation of the non-conservative terms
\begin{equation}
    B_{i+\frac{1}{2}} = diag(0,0,-u_{m,b},...,-u_{m,b}),
\end{equation}
at values
    $$
    u_{m,b} = \begin{cases}
	\frac{h_{r}^{2}u_{r}+h_{l}^{2}u_{l}+h_{l}h_{r}\left[(u_{l}-u_{r})log\left(\frac{h_{r}}{h_{l}}\right)-(u_{r}+u_{l})\right]}{(h_{r}-h_{l})^{2}} & \text{if} \ \ h_{r} \neq h_{l}, \\
	\frac{u_{r} + u_{l}}{2} & \text{if} \ \ h_{r} = h_{l}.
	\end{cases}
	$$
Setting $A_{i+\frac{1}{2}} = J_{i+\frac{1}{2}} + B_{i+\frac{1}{2}}$ and $S_{i+\frac{1}{2}}$, it can be shown that $\mathcal{A}_{i+\frac{1}{2}}$ is a Roe matrix in the sense of \cite{toumi1992weak}.\\

For the polynomial viscosity matrix $Q_{i+\frac{1}{2}}$ an HLL-like method that correspond to choosing a polynomial approximation of the matrix $Q$ as $P(x) = a_{0} + a_{1}x$ in (\ref{eq:PVM_method}) is used, see \cite{castro2017well} for more details. The coefficients are given as
$$a_{0} = \frac{S_{r}|S_l|-S_{l}|S_{r}|}{S_{r}-S_{l}}, \quad a_{1} = \frac{|S_r|-|S_{l}|}{S_{r}-S_{l}},$$
where $S_{r}$ and $S_{l}$ are the maximum and the minimum eigenvalue of $A_{i+\frac{1}{2}}$, respectively.

\begin{observacion}
    The eigenvalues of $A_{i+\frac{1}{2}}$ are computed numerically. However, it is possible to use the Cardano's formula to obtain exact eigenvalues.
\end{observacion}

\section{Numerical tests}
\label{sec:results}
In this section several tests with increasing complexity are considered to validate the results obtained starting from steady state initial conditions with the well-balanced first and second order schemes for the SWLME. Subsequently, we use a transient dam-break problem to compare the SWLME with the results obtained for the HSWME and the $\beta$HSWME, see \cite{Koellermeier2020}. 
For implementation details used in all examples of this section we refer to the implementation \cite{Pimentel2020}.

\subsection{Well-balanced property}
The first four test cases are intended to show that the scheme is effectively well-balanced. A 1000-point uniform mesh, free boundary conditions and a CFL number of $0.5$ are used. In all cases we exemplarily use $N=8$ moments and $g=9.812$.

\subsubsection*{Test 1: Lake at rest}
For the lake at rest, a zero velocity profile corresponding to water at rest with the following bottom topography is used in the spatial domain $[-1,1]$
\begin{equation}\label{eq:Test1_a}
    b_{0}(x) = \left\{\begin{array}{ll}
        2-x^{2} & \text{if} \ \ -0.5<x<0.5, \\
        1.75 & \text{otherwise},
\end{array}\right.
\end{equation}
and therefore
\begin{equation}\label{eq:Test1_b}
    W_{0}(x) = (h_{0}(x),u_{m,0}(x)h_{0}(x),\alpha_{1,0}(x)h_{0}(x),...,\alpha_{N,0}(x)h_{0}(x) ) = (3-b_{0}(x), 0, 0,...,0).
\end{equation}

The initial condition is shown in Figure \ref{fig:Test1_Initial_condition_hHH}. In Table \ref{tab:Error_Test1} we observe that the well-balanced and also the non well-balanced schemes of first and second order capture well the lake at rest. This is due to the fact that straight lines are used as the paths in the numerical scheme. This is a parameterization of the stationary solutions \cite{bermudez1994upwind,pares2004well}. For the first order test case, even the standard non well-balanced scheme gives the right solution 

\begin{table}[ht]
  	\centering
  	\begin{tabular}{|c|c|c|c|c|}
  		\hline 
  		Scheme (1000 cells) & $||\Delta h||_1$ (1st) & $||\Delta u||_1$ (1st) & $||\Delta h||_1$ (2nd) & $||\Delta u||_1$ (2nd) \\   
  		\hline 
  		Well-balanced & 0.00 & 8.16e-16 & 0.00 & 8.16e-16  \\ 
  		\hline 
  		Non well-balanced & 0.00 & 7.12e-16 & 4.51e-15 & 1.75e-14 \\ 
  		\hline 
  	\end{tabular} 
	  	\caption{Well-balanced vs non well-balanced schemes: $L^{1}$ errors $||\Delta \cdot||_1$ at time $t=0.5$ for the SWLME model with initial conditions (\ref{eq:Test1_a}) and (\ref{eq:Test1_b}).}
  	\label{tab:Error_Test1}
\end{table} 

\begin{figure}[htbp]
    \centering
        \includegraphics[width=0.55 \textwidth]{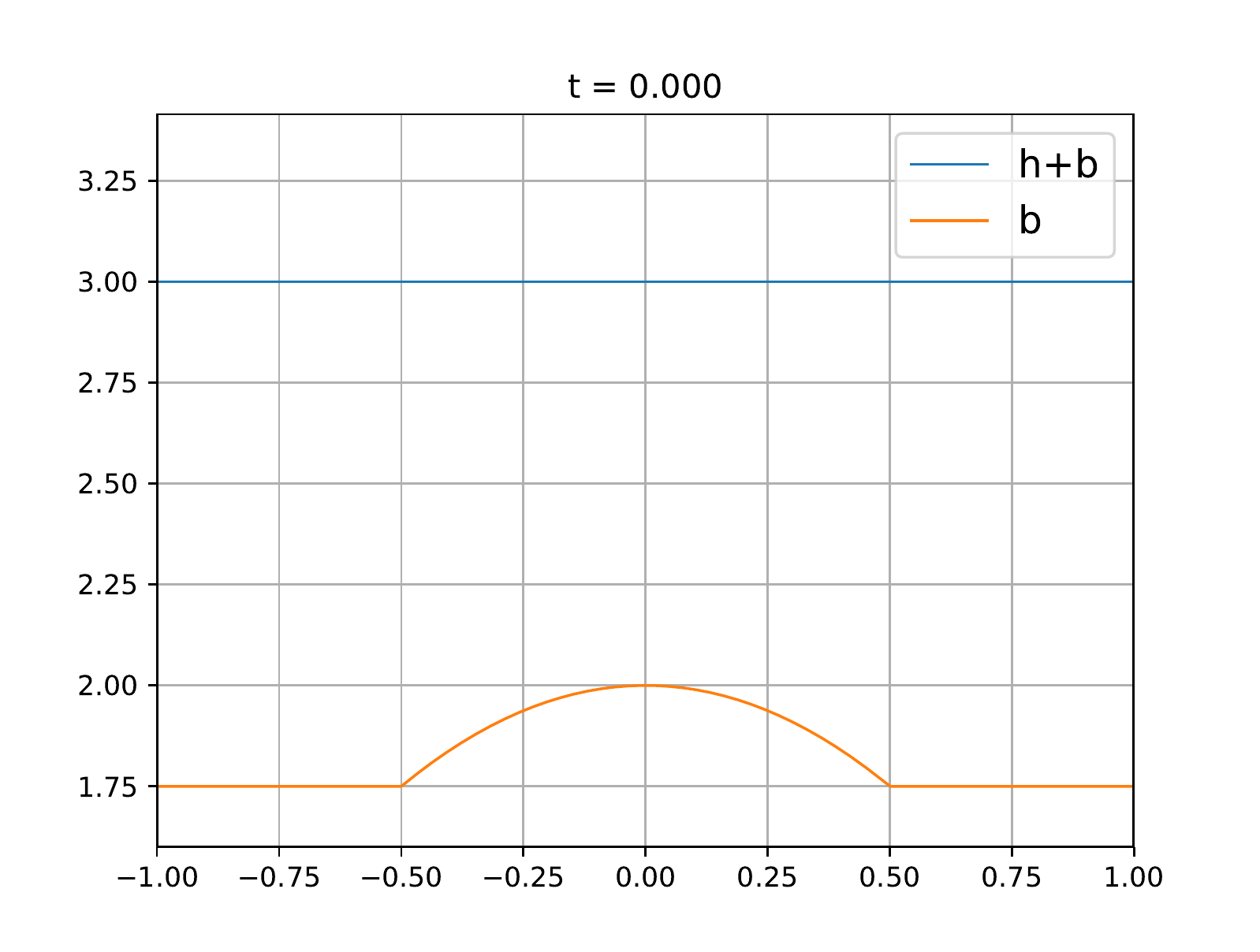}
        \caption{Initial condition for the lake at rest (\ref{eq:Test1_a}) and (\ref{eq:Test1_b}).}
        \label{fig:Test1_Initial_condition_hHH}
\end{figure}

\subsubsection*{Test 2: Subcritical stationary solution}
We consider a subcritical stationary solution as initial condition in the spatial domain $[0,3]$, similar to \cite{diaz2013high}. The bottom topography is chosen as
\begin{equation}\label{eq:Test2}
    b_{0}(x) = \left\{\begin{array}{ll}
        0.25(1+cos(5\pi(x+0.5))) & \text{if} \ \ 1.3<x<1.7, \\
        0 & \text{otherwise}.
    \end{array}\right.
\end{equation}
As $W_{0}(x)$ we take the subcritical stationary solution such that $C_1 = 3.5$, $C_2 = 17.56957396120237$ and $C_i = 0$ for $i\in\{3,...,N+2\}$. The initial condition is shown in Figure \ref{fig:Test2_Initial_condition}. In Table \ref{tab:Error_Test2} we observe that our well-balanced schemes of first and second order capture well the subcritical stationary solution while the non well-balanced schemes do not. The non well-balanced scheme shows a clear error whereas the well-balanced scheme is exact up to almost machine prevision.

\begin{table}[ht]
  	\centering
  	\begin{tabular}{|c|c|c|c|c|}
  		\hline 
  		Scheme (1000 cells) & $||\Delta h||_1$ (1st) & $||\Delta u||_1$ (1st) & $||\Delta h||_1$ (2nd) & $||\Delta u||_1$ (2nd) \\  
  		\hline 
  		Well-balanced & 9.16e-16 & 1.79e-15 & 1.42e-15 & 3.24e-15  \\ 
  		\hline 
  		Non well-balanced & 2.48e-6 & 5.08e-6 & 3.21e-5 & 8.40e-5 \\ 
  		
  		\hline 
  	\end{tabular} 
	  	\caption{Well-balanced vs non well-balanced schemes: $L^{1}$ errors $||\Delta \cdot||_1$ at time $t=0.5$ for the SWLME model with initial condition (\ref{eq:Test2}).}

  	\label{tab:Error_Test2}
\end{table}

\begin{figure}[h]
		\begin{subfigure}{0.5\textwidth}
			\includegraphics[width=1.1\linewidth]{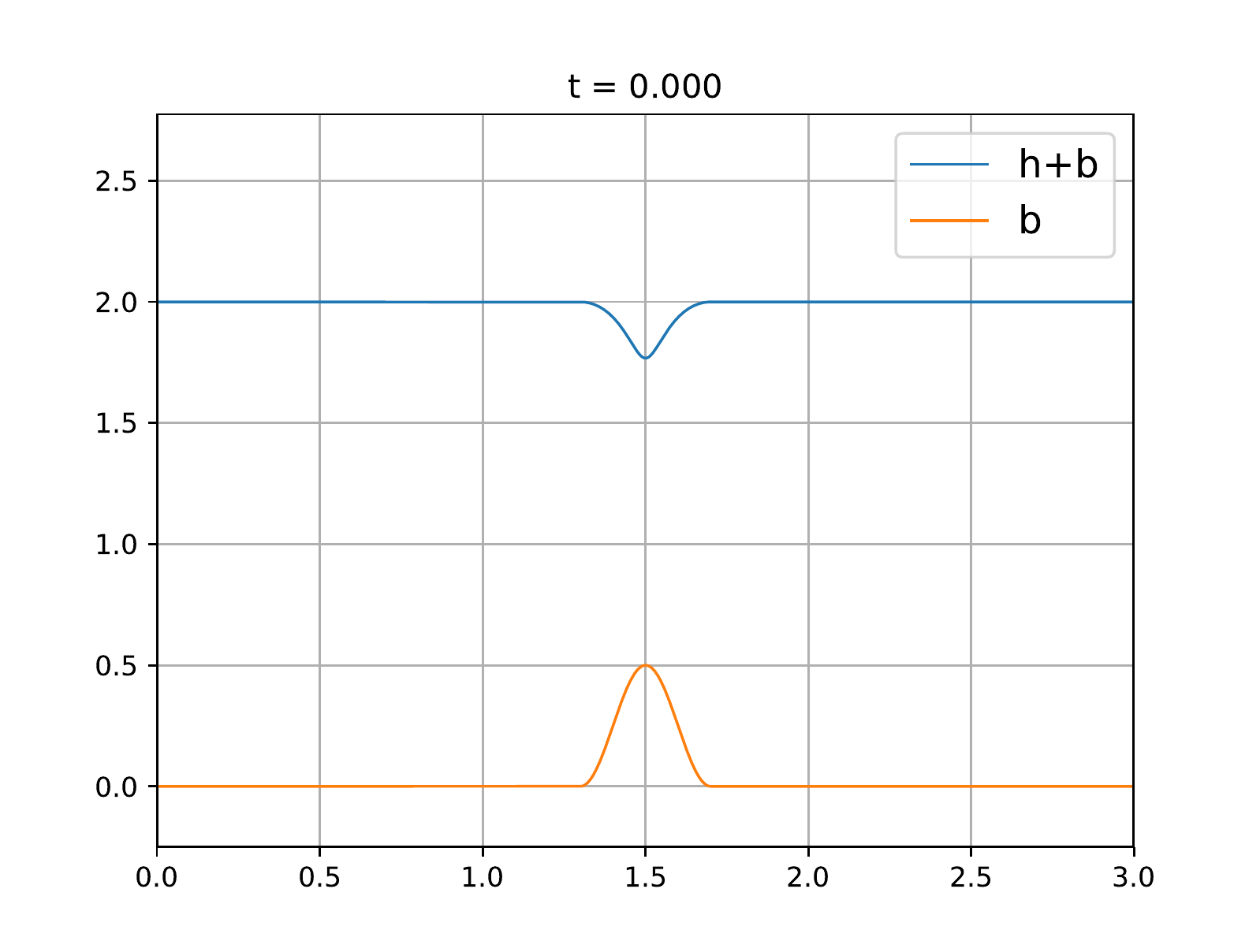}
		\end{subfigure}
		\begin{subfigure}{0.5\textwidth}
				\includegraphics[width=1.1\linewidth]{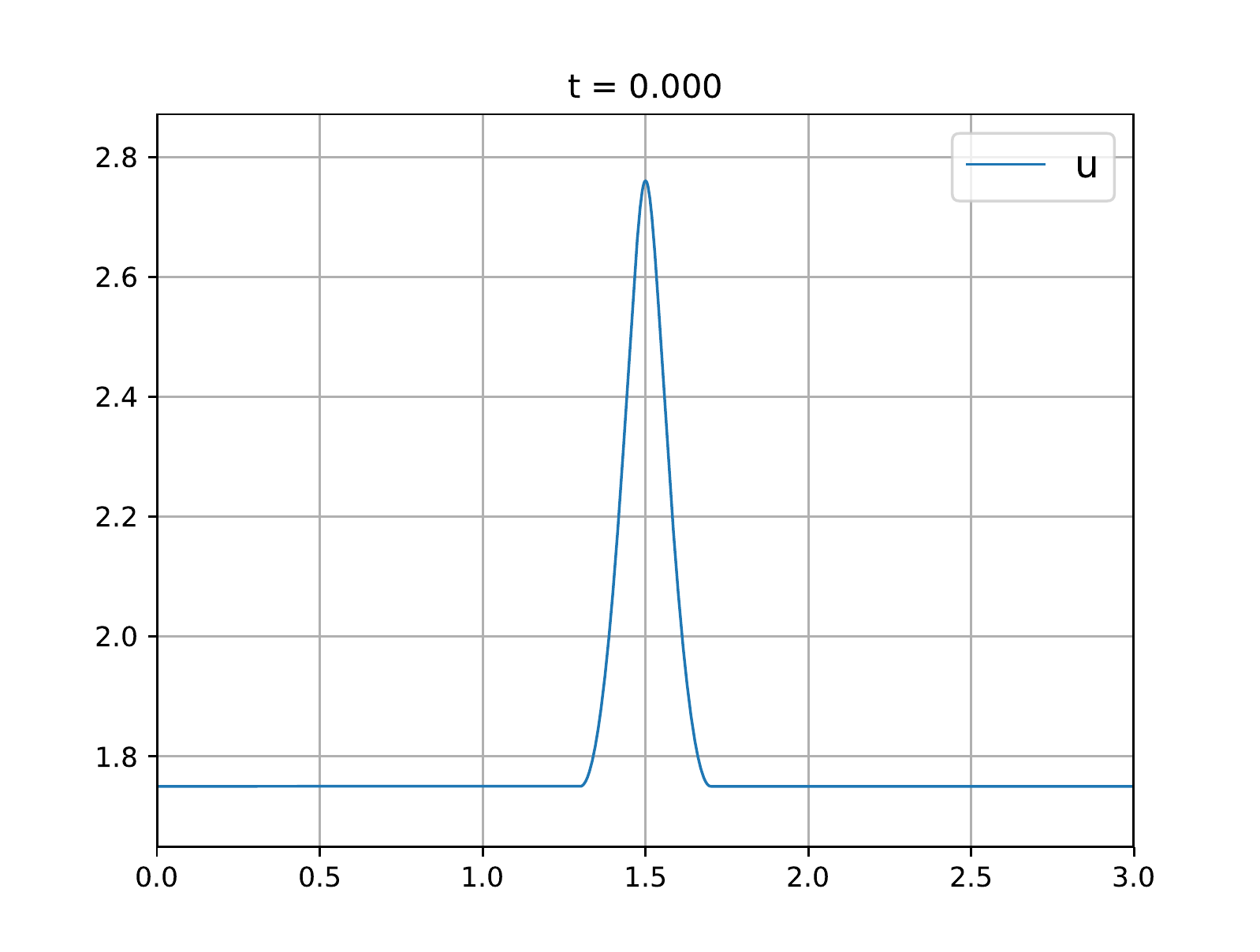}
		\end{subfigure}
		\caption{Initial condition for the subcritical stationary solution (\ref{eq:Test2}).}
		\label{fig:Test2_Initial_condition}
\end{figure}

\subsubsection*{Test 3: Transcritical stationary solution}
Next, we consider a transcritical stationary solution using an initial condition in $[0,3]$ similar to  \cite{diaz2013high}. The bottom topography is chosen as
\begin{equation}\label{eq:Test3_a}
    b_{0}(x) = \left\{\begin{array}{ll}
        0.25(1+cos(5\pi(x+0.5))) & \text{if} \ \ 1.3<x<1.7, \\
        0 & \text{otherwise}.
\end{array}\right.
\end{equation}
As $W_{0}(x)$ we take the transcritical stationary solution 
\begin{equation}\label{eq:Test3_b}
    W_{0}(x) = \left\{\begin{array}{ll}
        W_{*}(x) & \text{if} \ \ x<1.5 \\
        W^{*}(x) & \text{if} \ \ x>1.5
\end{array}\right.
\end{equation}
where $W_{*}$ and $W^{*}$ are the subcritical and supercritical stationary solutions such that $C_1 = 2.5$, $C_2 = 21,15525$ and $C_i = 0$ for $i\in\{3,...,N+2\}$. The initial condition is shown in Figure \ref{fig:Test3_Initial_condition}. In Table \ref{tab:Error_Test3} we observe that our well-balanced schemes of first and second order capture well the transcritical stationary solution while the non well-balanced schemes do not. Again, the non well-balanced schemes result in a large error while the well-balanced schemes achieve a very accurate steady state solution.

\begin{table}[ht]
  	\centering
  	\begin{tabular}{|c|c|c|c|c|}
  		\hline 
  		Scheme (1000 cells) & $||\Delta h||_1$ (1st) & $||\Delta u||_1$ (1st) & $||\Delta h||_1$ (2nd) & $||\Delta u||_1$ (2nd) \\ 
  		\hline 
  		Well-balanced & 3.53e-14 & 2.95e-13 & 3.53e-14 & 2.98e-13  \\ 
  		\hline 
  		Non well-balanced & 1.46e-5 & 1.22e-4 & 3.07e-4 & 1.12e-3 \\ 
  		
  		\hline 
  	\end{tabular} 
	  	\caption{Well-balanced vs non well-balanced schemes: $L^{1}$ errors $||\Delta \cdot||_1$ at time $t=0.5$ for the SWLME model with initial condition (\ref{eq:Test3_a}) and (\ref{eq:Test3_b}).}
  	\label{tab:Error_Test3}
\end{table}

\begin{figure}[h]
		\begin{subfigure}{0.5\textwidth}
			\includegraphics[width=1.1\linewidth]{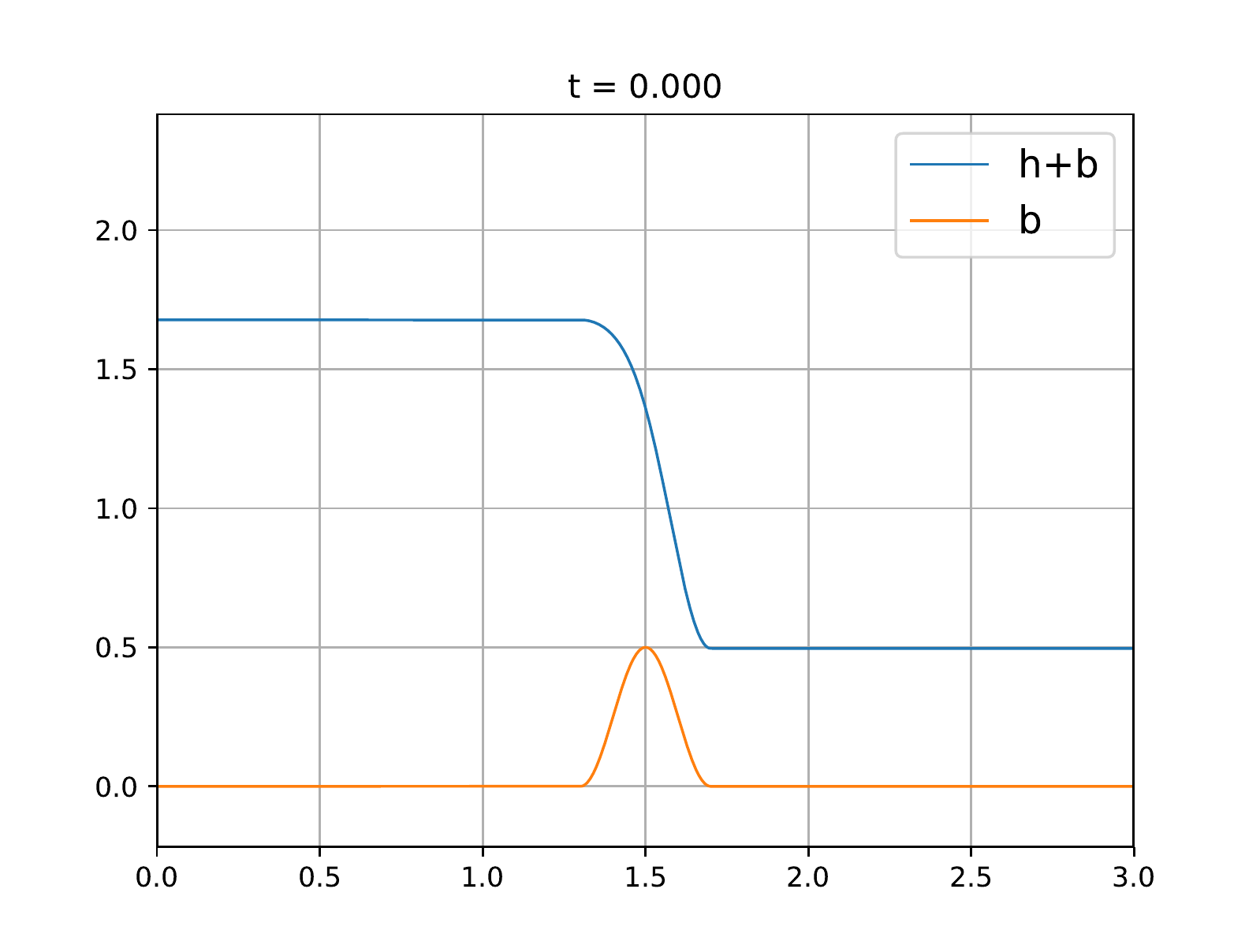}
		\end{subfigure}
		\begin{subfigure}{0.5\textwidth}
				\includegraphics[width=1.1\linewidth]{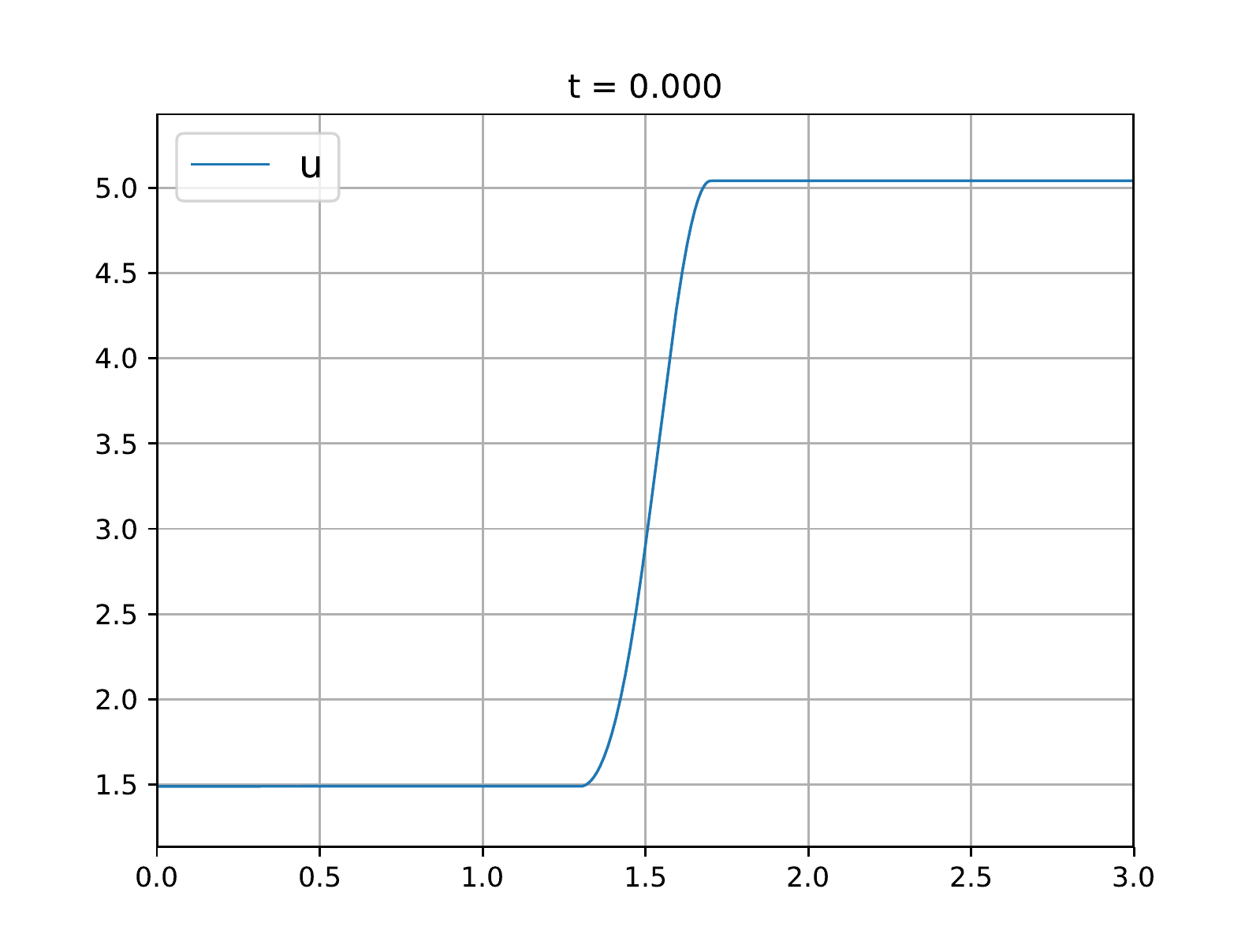}
		\end{subfigure}
		\caption{Initial condition for the subcritical stationary solution (\ref{eq:Test3_a}) and (\ref{eq:Test3_b}). }
		\label{fig:Test3_Initial_condition}
\end{figure}

\subsubsection*{Test 4: Subcritical stationary solution with non zero moments}
Lastly, we consider the following initial condition in $[0,3]$ that is a subcritical stationary solution with non-vanishing coefficients $\alpha_i$. The bottom topography is chosen as
\begin{equation}\label{eq:Test4}
    b_{0}(x) = \left\{\begin{array}{ll}
        0.25(1+cos(5\pi(x+0.5))) & \text{if} \ \ 1.3<x<1.7 \\
        0 & \text{otherwise}
\end{array}\right.
\end{equation}
As $W_{0}(x)$ we use the subcritical stationary solution such that $C_1 = 3.5$, $C_2 = 21,15525$ and $C_i = 0.25$ for $i\in\{3,...,N+2\}$. The initial condition is shown in Figure \ref{fig:Test4_Initial_condition}. In Table \ref{tab:Error_Test4} we observe that our well-balanced schemes of first and second order capture well the subcritical stationary solution while the non well-balanced schemes do not. Even in this test case with non-zero higher-order coefficients $\alpha_i$ the well-balanced scheme is much more accurate than the standard non well-balanced version.
\begin{table}[ht]
  	\centering
  	\begin{tabular}{|c|c|c|c|c|c|c|}
  		\hline 
  		Scheme & $||\Delta h||_1$, 1st & $||\Delta u||_1$ (1st) & $|| \Delta \alpha_{i}||_1$ (1st) & $||\Delta h||_1$ (2nd)& $||\Delta u||_1$ (2nd) & $|| \Delta \alpha_{i}||_1$ (2nd) \\ 
  		\hline 
  		wb & 4.00e-15 & 9.71e-15 & 4.45e-15 & 2.56e-15 & 7.66e-15 & 5.04e-15 \\ 
  		\hline 
  		Non wb & 3.11e-6 & 6.65e-6 & 6.98e-7 & 3.98e-5 & 1.04e-4 & 2.52e-5 \\ 
  		
  		\hline 
  	\end{tabular} 
	  	\caption{Well-balanced (WB) vs non well-balanced schemes: $L^{1}$ errors $||\Delta \cdot||_1$ at time $t=0.5$ for the SWLME model with initial condition (\ref{eq:Test4}).}

  	\label{tab:Error_Test4}
\end{table}

\begin{figure}[h]
		\begin{subfigure}{0.5\textwidth}
			\includegraphics[width=1.1\linewidth]{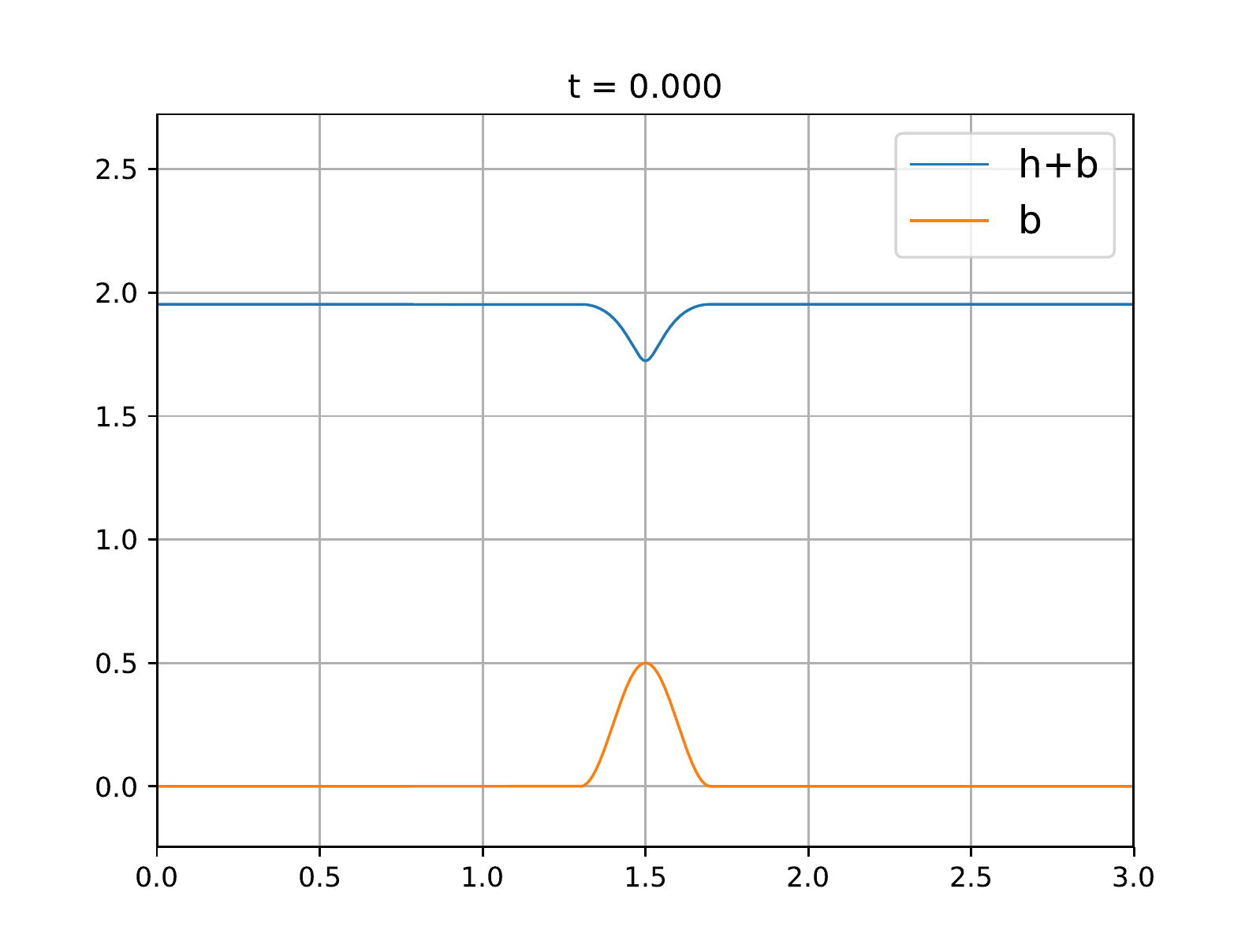}
		\end{subfigure}
		\begin{subfigure}{0.5\textwidth}
				\includegraphics[width=1.1\linewidth]{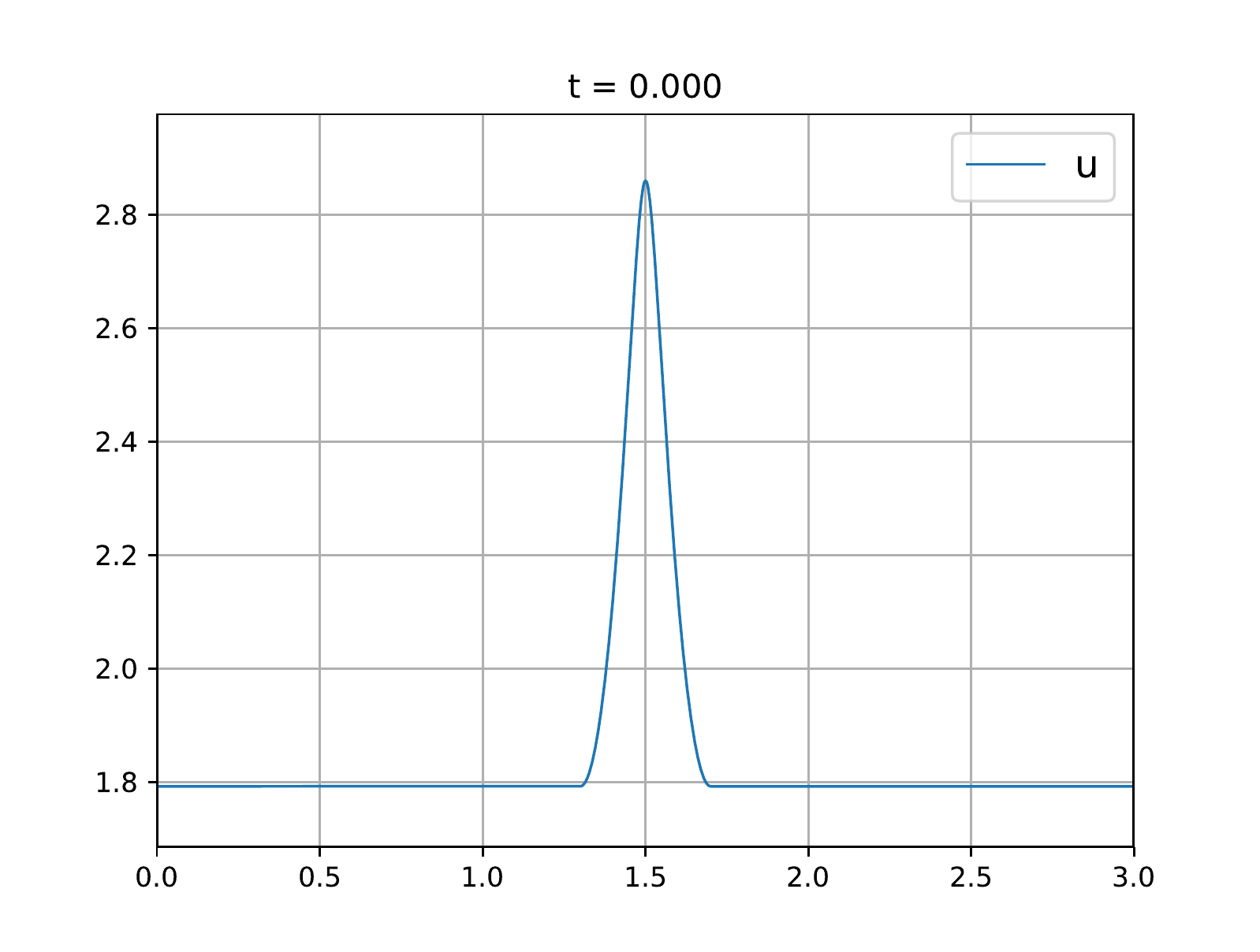}
		\end{subfigure}
		\center{\begin{subfigure}{0.5\textwidth}
				\includegraphics[width=1.1\linewidth]{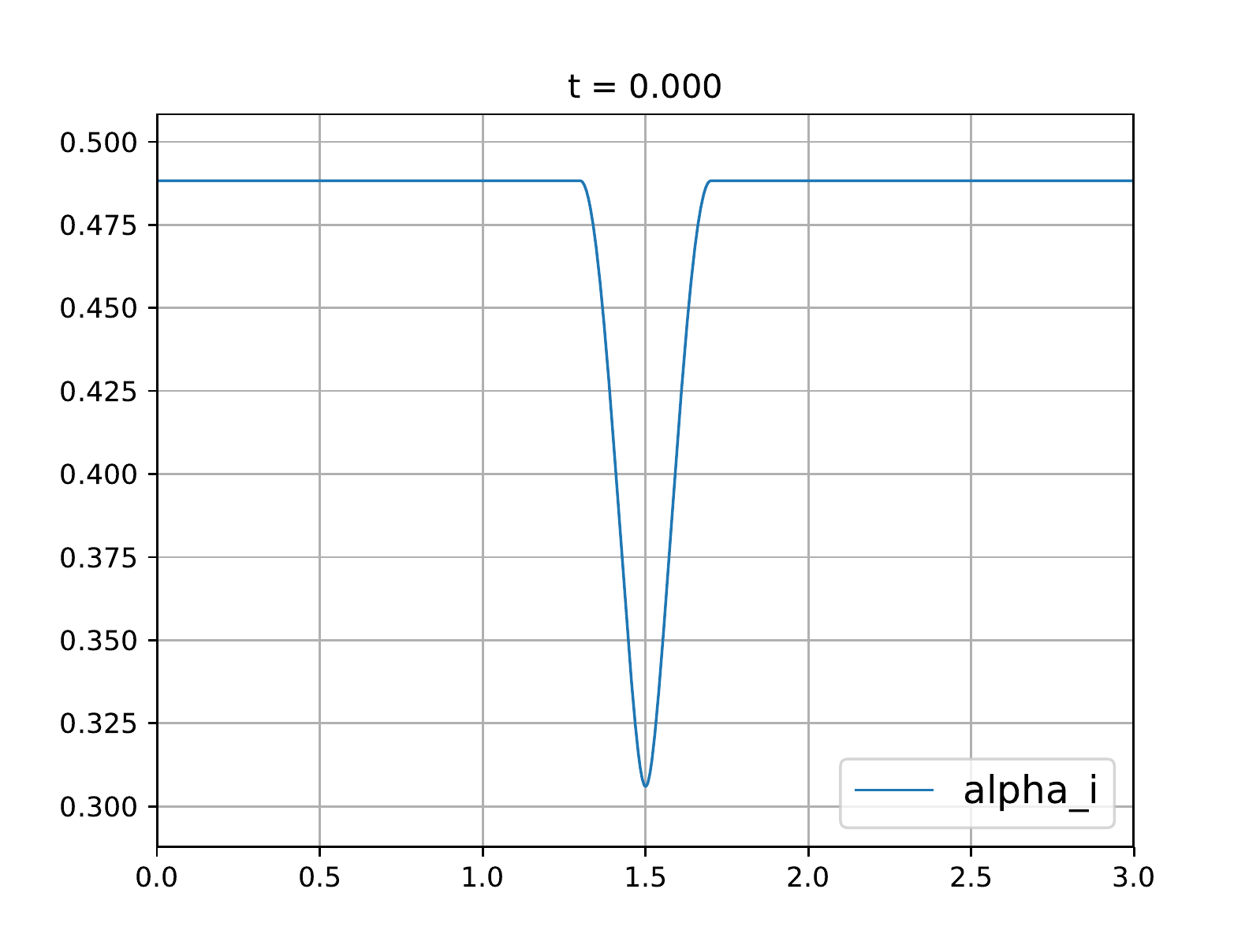}
		\end{subfigure}}
		\caption{Initial condition for the subcritical stationary solution (\ref{eq:Test4}).}
		\label{fig:Test4_Initial_condition}
\end{figure}

\subsection{Comparison between the SWLME, HSWME and $\beta$HSWME}
In the following two tests, the results for the new SWLME model are compared with other hyperbolic models, HSWME and $\beta$HSWME, for which a Roe matrix was derived and explicitly given in the appendix \ref{app}. These tests will be done in the spatial domain $[-0.4,0.4]$ with $g=1$ and $N=8$. We consider a flat bottom topography ($b_x=0$) and neglect friction terms. In this test case, the well-balanced property is of no interest, we therefore only compare the standard first and second order schemes.

\subsubsection*{Test 5: transient model comparison with standard dam-break test}
We are going to consider the following dam-break initial condition taken from \cite{Koellermeier2020} without friction terms
\begin{equation}\label{eq:Test5_a}
    W_{0}(x) = (h_{0}(x),u_{m,0}(x)h_{0}(x),\alpha_{1,0}(x)h_{0}(x),...,\alpha_{N,0}(x)h_{0}(x) ), 
\end{equation}
where $u_{m,0}(x) = 0.25$, $\alpha_{1,0}(x)=-0.25$, $\alpha_{N,0}(x)=0.25$, $\alpha_{i,0}(x)=0, i\in\{2,...,N-1\}$, and
\begin{equation}\label{eq:Test5_b}
    h_{0}(x) = \left\{\begin{array}{ll}
        5 & \text{if} \ \ x<0, \\
        1 & \text{if} \ \ x>0.
        \end{array}\right.
\end{equation}

Figure \ref{O1_noWB_SWLME_vs_noWB_HSWME_vs_noWB_betaHSWME_Test5_t01} shows the numerical results obtained with the first and second order scheme for the SWLME, and the first order schemes for the HSWME and the $\beta$HSWME. The results for the second order schemes applied to the latter two models are quantitatively the same as the first order results and thus omitted here. We can conclude that the results obtained are quite similar for all models in the variables $h,u$ and $\alpha_1$. As expected, the second order scheme captures the rarefaction waves better. We point out that the speed of the shock that travels from the left to the right is slightly higher in the case of the SWLME than in the other two models because in (\ref{SWLME_eigenvalues}) we observe that all the $\alpha_{i}$ are taken into account for the maximum and minimum eigenvalues while in the HSWME and $\beta$HSWME only $\alpha_{1}$ is contributing.

\begin{figure}[h]
		\begin{subfigure}{0.5\textwidth}
			\includegraphics[width=1.1\linewidth]{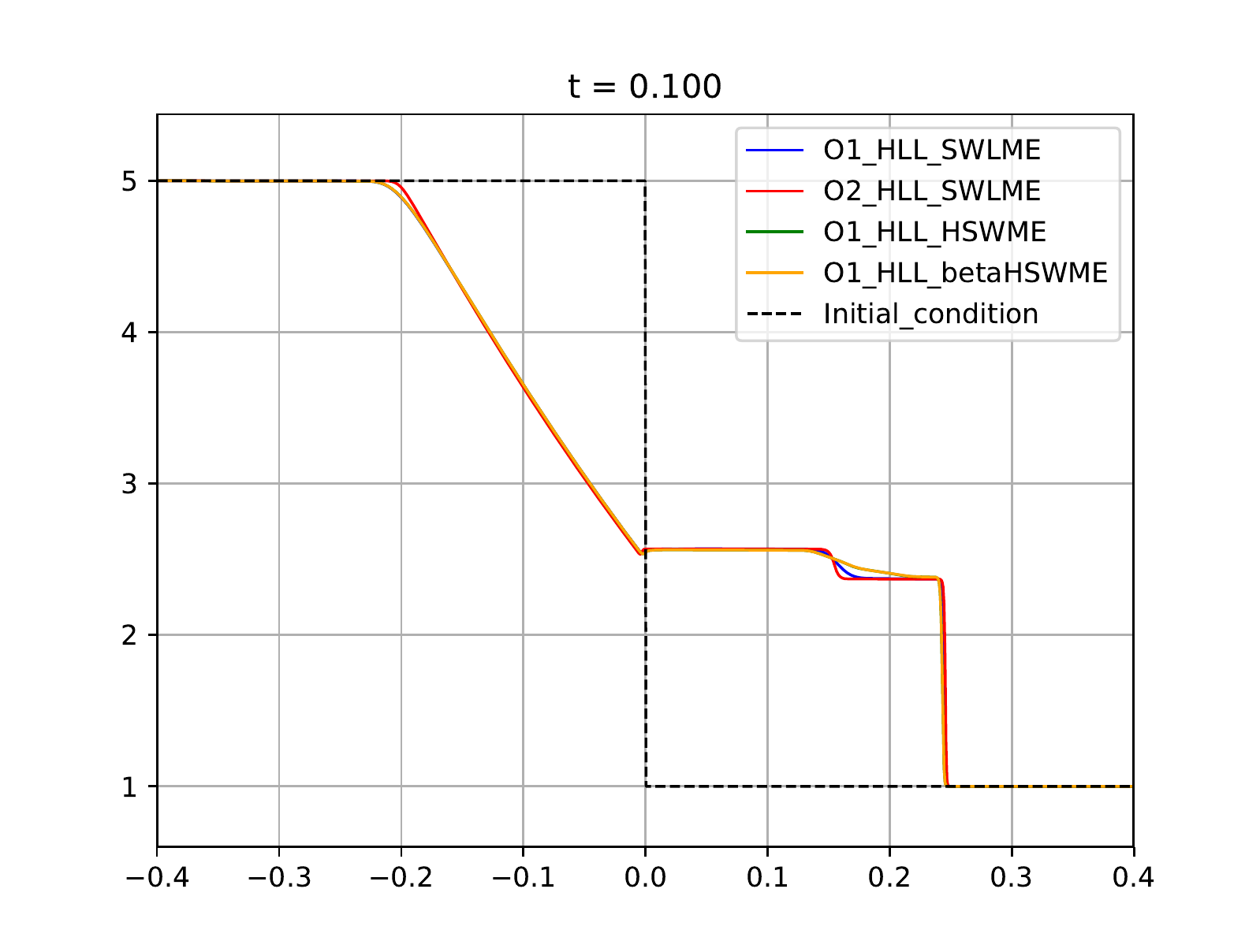}
			\caption{Water height $h$.}
		\end{subfigure}
		\begin{subfigure}{0.5\textwidth}
				\includegraphics[width=1.1\linewidth]{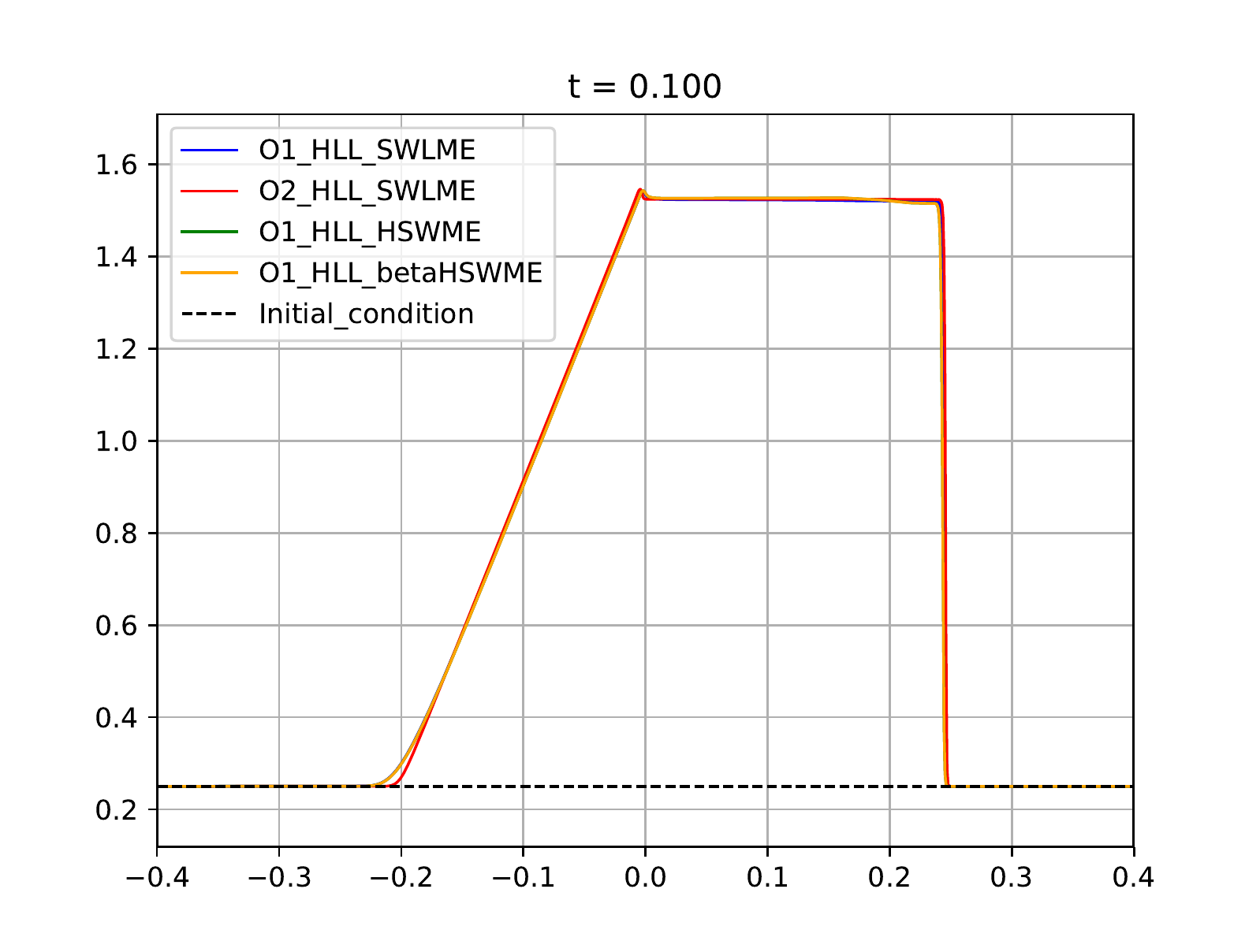}
				\caption{Velocity $u$.}
		\end{subfigure}
		\newline
		\begin{subfigure}{0.5\textwidth}
				\includegraphics[width=1.1\linewidth]{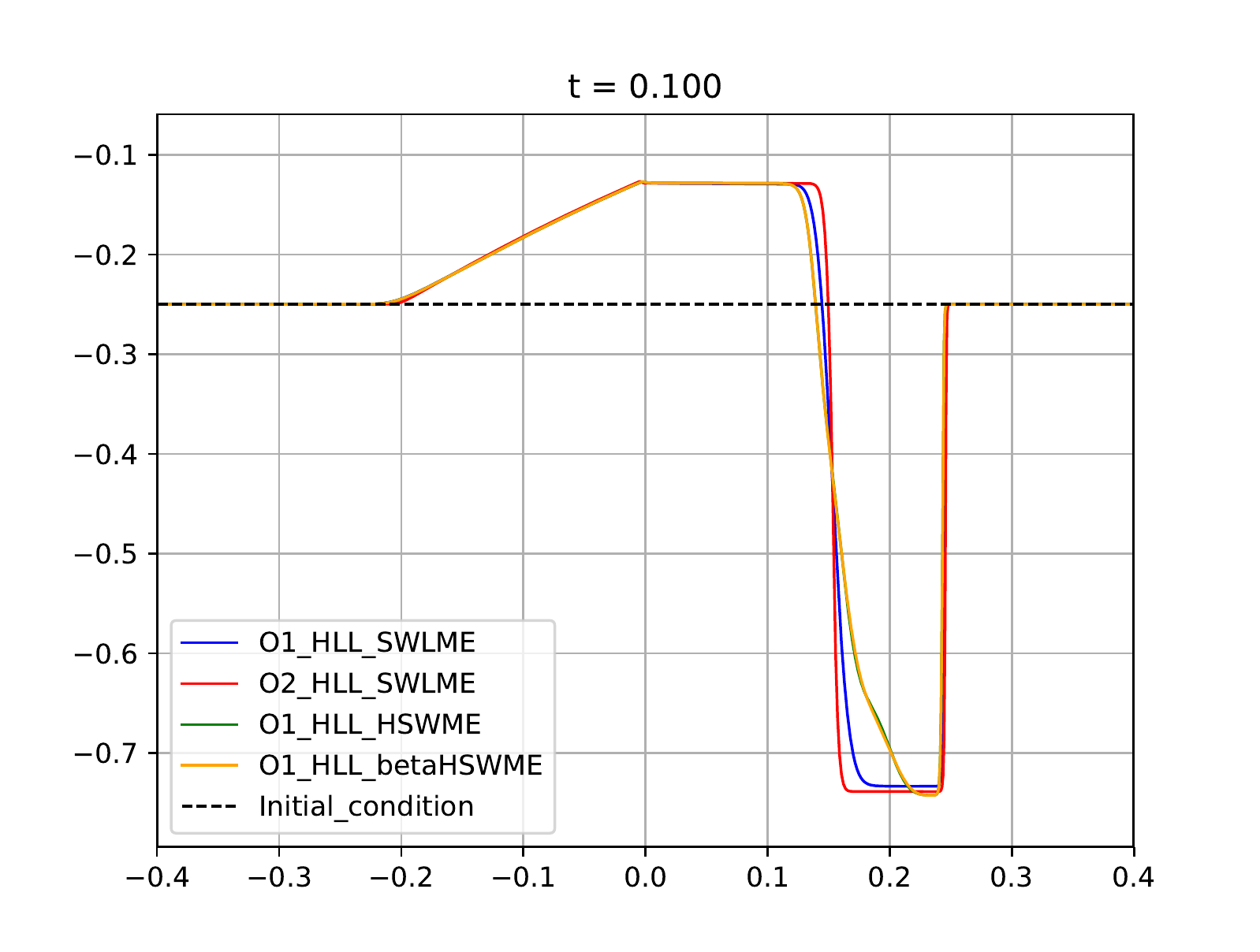}
				\caption{First coefficient $\alpha_1$.}
		\end{subfigure}
		\begin{subfigure}{0.5\textwidth}
				\includegraphics[width=1.1\linewidth]{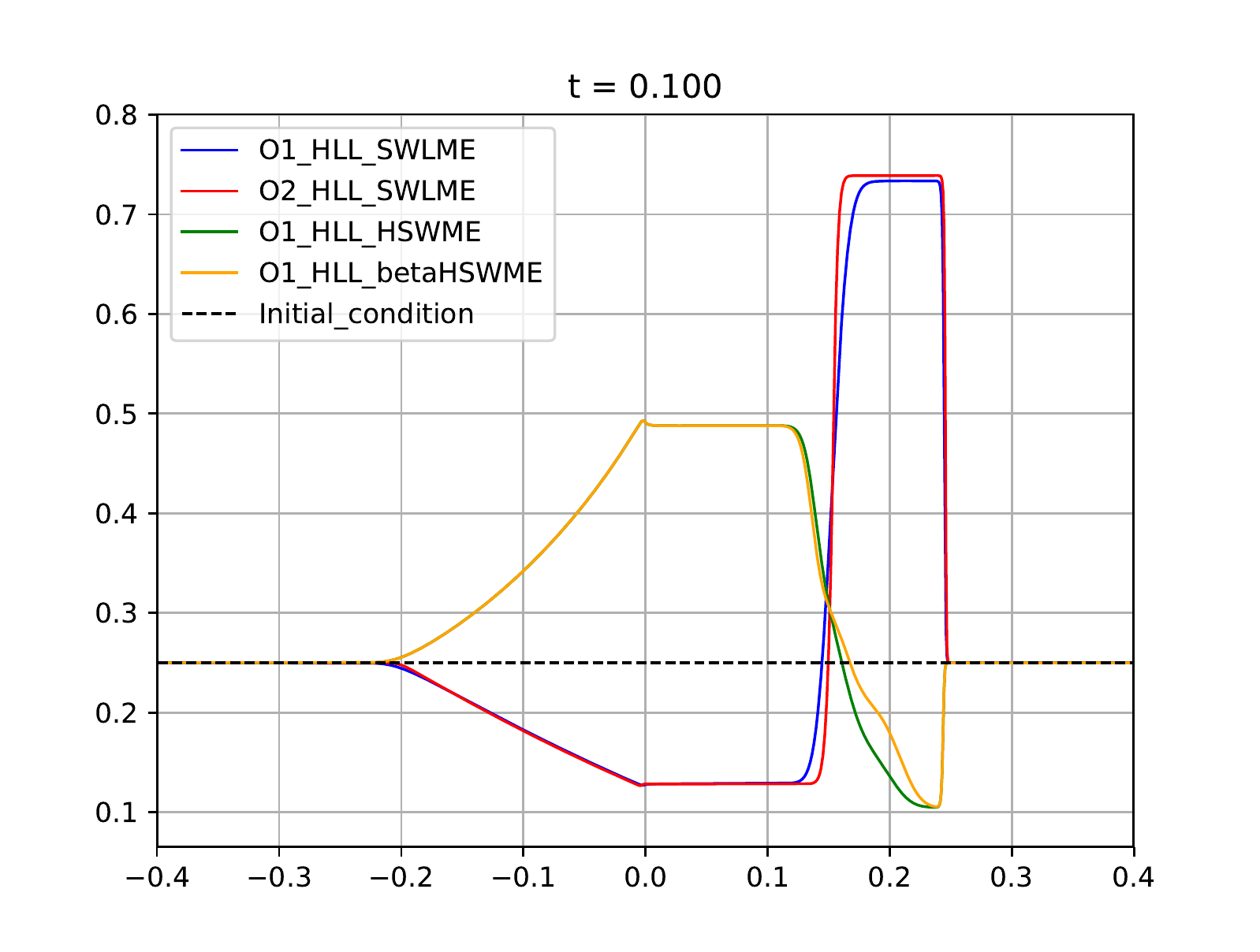}
				\caption{Last coefficient $\alpha_8$.}
		\end{subfigure}
		\caption{Results obtained with the different models for the standard dam-break initial condition (\ref{eq:Test5_a}) and (\ref{eq:Test5_b}) for variables $h,u,\alpha_1$, $\alpha_{8}$ at t=0.1.}
		\label{O1_noWB_SWLME_vs_noWB_HSWME_vs_noWB_betaHSWME_Test5_t01}
\end{figure}

\subsubsection*{Test 6: transient model comparison with square root velocity profile}
For the last test, we consider the following dam-break initial condition:
\begin{equation}\label{eq:Test6_a}
    W_{0}(x) = (h_{0}(x),u_{m,0}(x)h_{0}(x),\alpha_{1,0}(x)h_{0}(x),...,\alpha_{N,0}(x)h_{0}(x) ), 
\end{equation}
where we use a square root initial velocity profile \eqref{expansion} $u(0,x,\zeta)=u_m(0,x)+\sum_{j=1}^{N}\alpha_j(0,x)\phi_j(\zeta) = \sqrt{\zeta}$, such that the initial variables can be computed according to \eqref{e:IC_u} and \eqref{e:IC_alpha} as $u_{m,0}(x) = 1$ and
\begin{equation}
\begin{array}{cccc} 
    \alpha_{1,0}(x)=-\frac{3}{5}, & \alpha_{2,0}(x)=-\frac{1}{7}, & \alpha_{3,0}(x)=-\frac{1}{15}, & \alpha_{4,0}(x)=-\frac{3}{77},\\ 
    \alpha_{5,0}(x)=-\frac{1}{39}, & \alpha_{6,0}(x)=-\frac{1}{55}, & \alpha_{7,0}(x)=-\frac{3}{221}, & \alpha_{8,0}(x)=-\frac{1}{95}.
\end{array}
\end{equation}
The initial water height is chosen as
\begin{equation}\label{eq:Test6_b}
    h_{0}(x) = \left\{\begin{array}{ll}
        5 & \text{if} \ \ x<0, \\
        1 & \text{if} \ \ x>0.
        \end{array}\right.
\end{equation}

In Figure \ref{O1_noWB_SWLME_vs_noWB_HSWME_vs_noWB_betaHSWME_Test6_t01} we show the numerical results obtained with the first and second order scheme for the SWLME, and the first order schemes for the HSWME and the $\beta$HSWME. We can conclude that the results obtained are quite similar for all of them in the variables $h,u$ and $\alpha_1$. This is not the case for the variable $\alpha_8$ where we can see that both the HSWME and the $\beta$HSWME result in strong oscillations. In comparison, the new SWLME is more stable than the other two models. Again the second order scheme captures the rarefaction waves better. Note that the emerging instability is not the result of an unstable high-order scheme, as the solutions for HSWME and $\beta$HSWME are even unstable with the first order scheme, while the SWLME yields stable results for both schemes. We point out that in this test the difference between the speed of the shock is even higher in the SWLME than in the other test because this time all the $\alpha_i$ have a non-zero initial value.
\begin{figure}[h]
		\begin{subfigure}{0.5\textwidth}
			\includegraphics[width=1.1\linewidth]{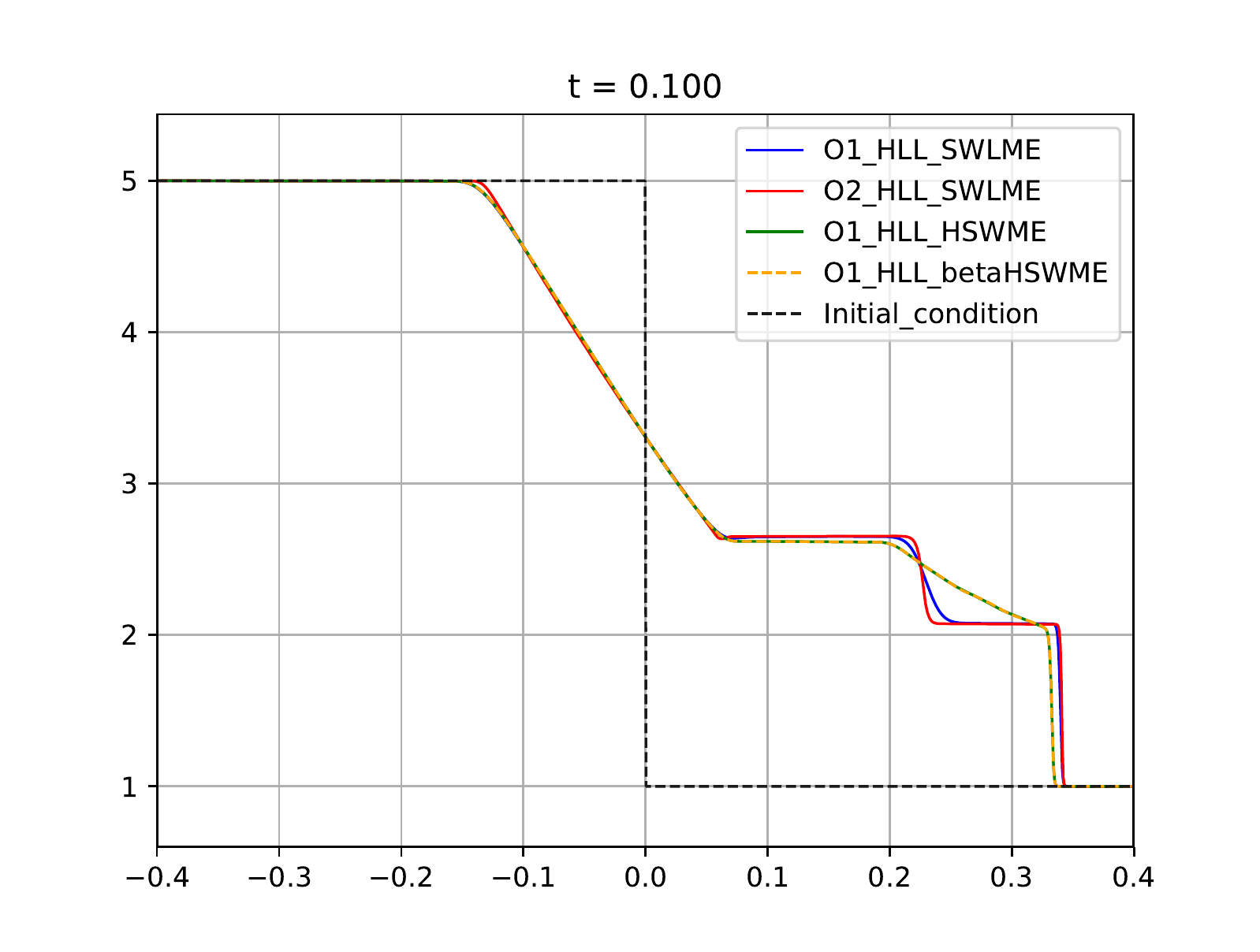}
			\caption{Water height $h$.}
		\end{subfigure}
		\begin{subfigure}{0.5\textwidth}
				\includegraphics[width=1.1\linewidth]{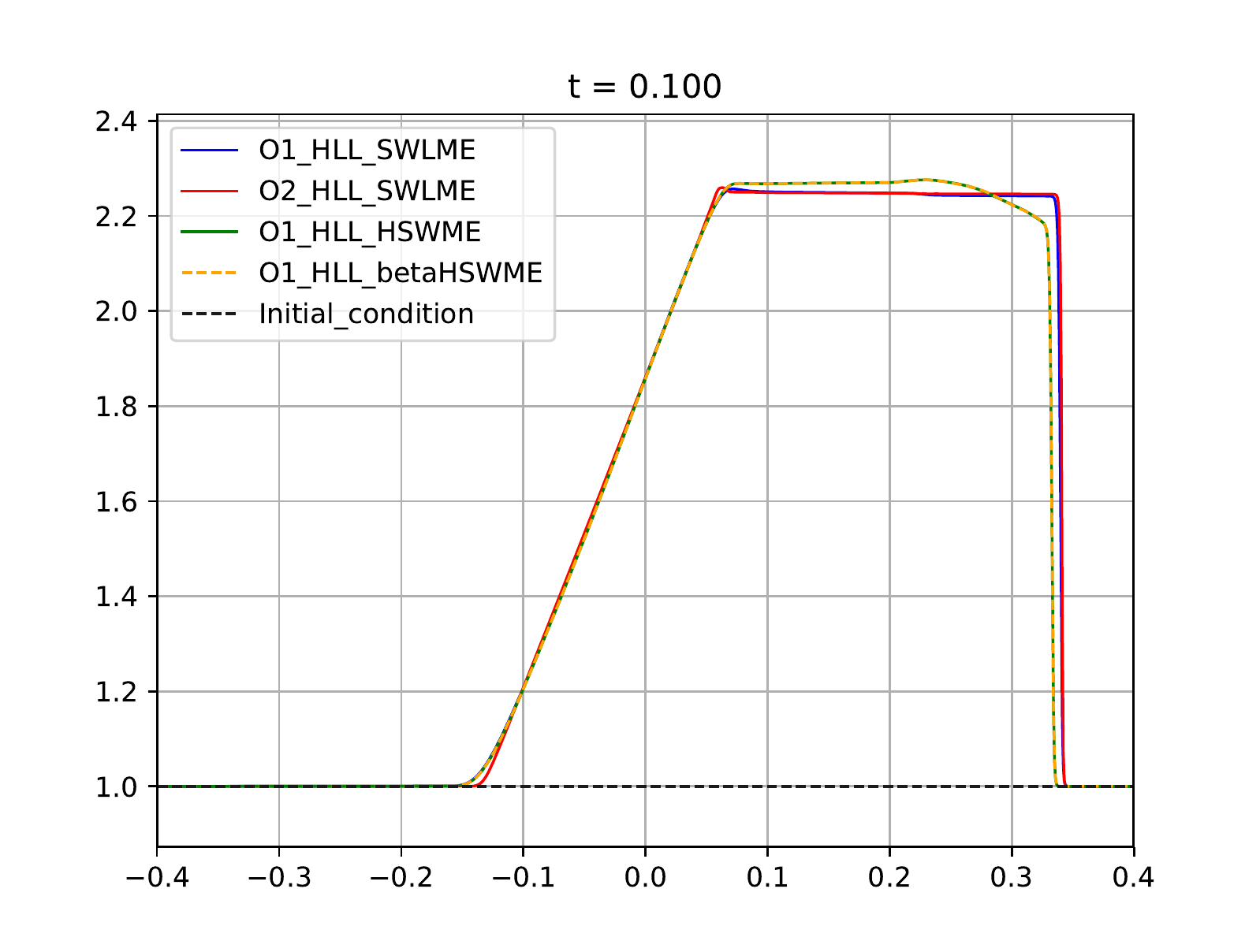}
				\caption{Velocity $u$.}
		\end{subfigure}
		\newline
		\begin{subfigure}{0.5\textwidth}
				\includegraphics[width=1.1\linewidth]{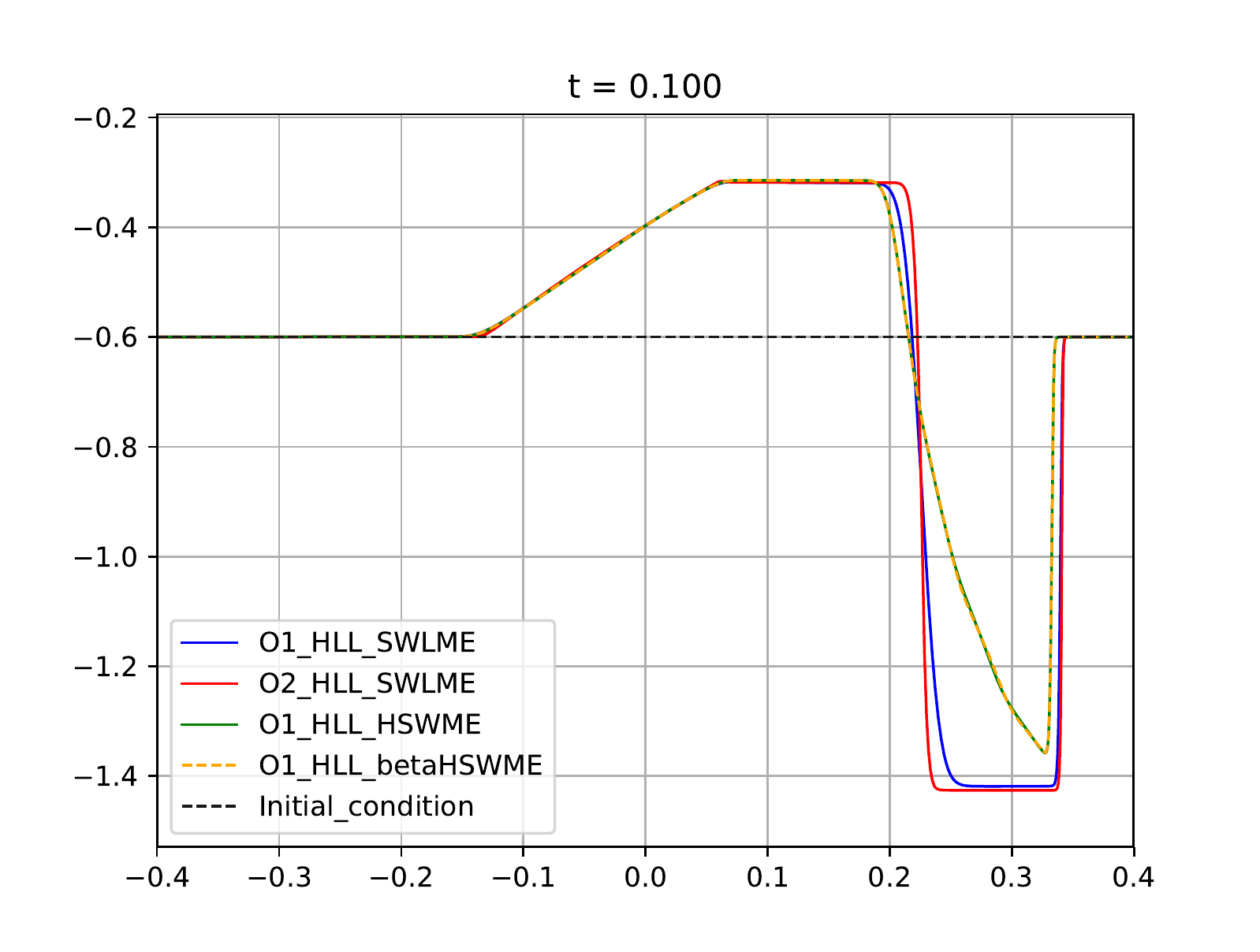}
				\caption{First coefficient $\alpha_1$.}
		\end{subfigure}
		\begin{subfigure}{0.5\textwidth}
				\includegraphics[width=1.1\linewidth]{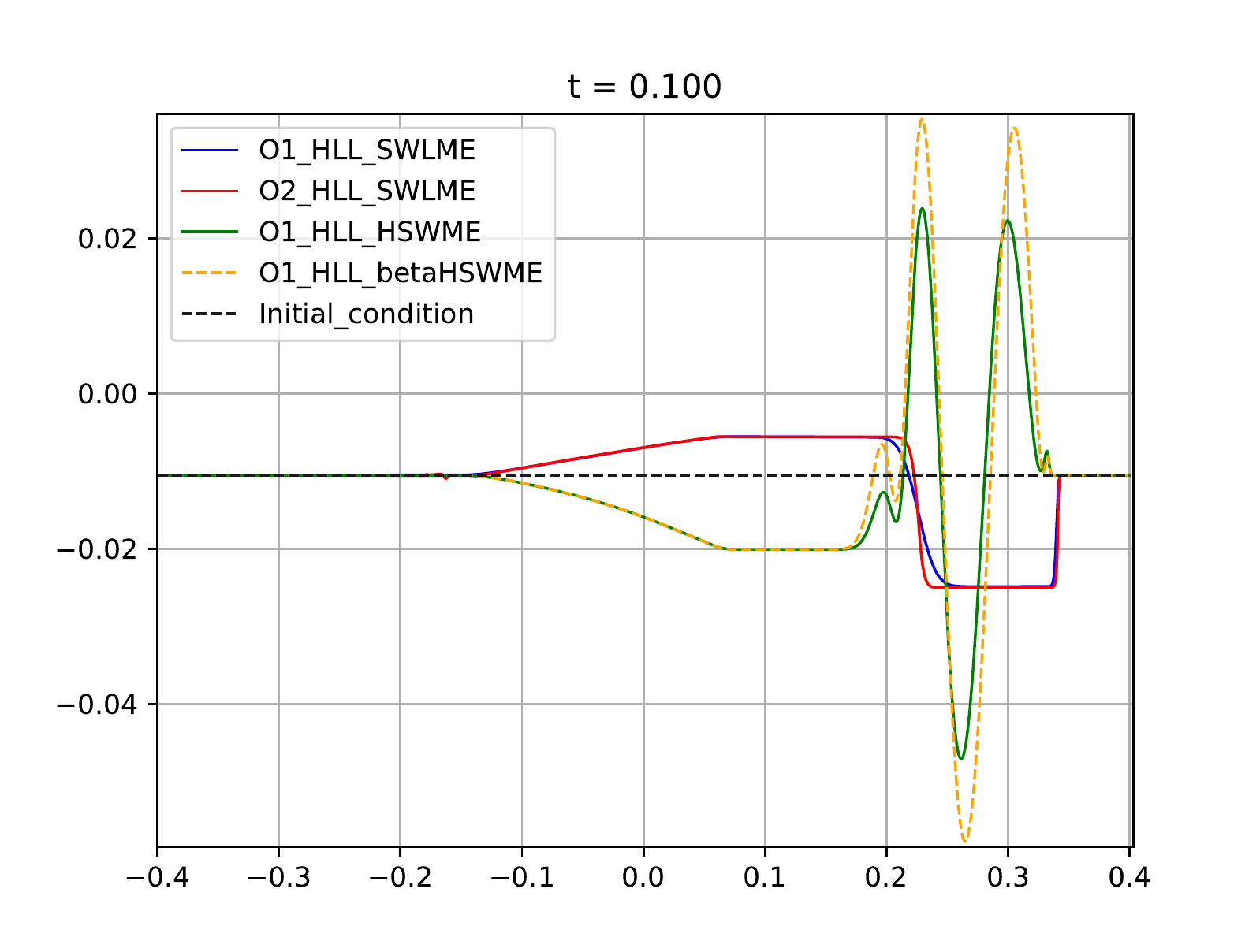}
				\caption{Last coefficient $\alpha_8$.}
		\end{subfigure}
		\caption{Results obtained with the different models for the dam-break with square root velocity profile initial condition (\ref{eq:Test6_a}) and (\ref{eq:Test6_b}) for variables $h,u,\alpha_1$, $\alpha_{8}$ at t=0.1.}
		\label{O1_noWB_SWLME_vs_noWB_HSWME_vs_noWB_betaHSWME_Test6_t01}
\end{figure}

\section{Conclusion}
In this paper, we analytically and numerically investigate steady states of Shallow Water Moment Equations (SWME). After showing that the steady states for the SWME with $N=1$ are extensions of the standard Shallow Water Equations (SWE), we pointed out that the case for arbitrary $N$ poses difficulties due to the loss of hyperbolicity and the structure of the model. The analysis was generalized with the help of a newly derived model called Shallow Water Linearized Moment Equations (SWLME), based on a linearization during the derivation. The concise derivation of the SWLME allowed to prove hyperbolicity and to fully characterize its eigenstructure analytically. This information was used to define a first order and a second order well-balanced numerical scheme preserving the steady states of the model numerically up to machine precision. Numerical results for lake-at-rest, subcritical, and transcritical initial conditions showed the success of the numerical scheme. Additionally, we compared the new SWLME model to other existing shallow water moment models, obtaining very similar solutions for the standard dam-break test. The solution for a more complex velocity profile seems more stable with the new SWLME model while existing models show emerging instabilities. 

The current work is a major step towards a better understanding of shallow water moment models and opens up many possibilities for future work and applications. Next steps could be a detailed stability analysis of the models including the right hand side friction terms, which were neglected in this paper, or the design of proper implicit numerical scheme for potentially stiff friction terms. An extension towards well-balanced schemes of higher-order is possible following the construction of the second order scheme in this paper.

\section*{Acknowledgements}

The authors are thankful to Manuel J. Castro D\'iaz and Carlos Par\'es for the useful suggestions and comments on this work.

This research has been partially supported by the European Union's Horizon 2020 research and innovation program under the Marie Sklodowska-Curie grant agreement no. 888596. J. Koellermeier is a postdoctoral fellow in fundamental research of the Research Foundation - Flanders (FWO), funded by FWO grant no. 0880.212.840. Ernesto Pimentel acknowledges financial support from the Spanish Government-FEDER funded project MEGAFLOW (RTI2018-096064-B-C21), the Junta de Andaluc\'ia-FEDER-University of Málaga funded project UMA18-Federja-161.

\appendix
\section{HSWME and $\beta$-HSWME models}\label{app}
The HSWME and $\beta$-HSWME models are derived and explicitly given in \cite{Koellermeier2020}. 
For our numerical schemes, we can write these two models in the form \eqref{FBSsystem} where the conservative flux is given by
$$F^{HSWME}(U)=F^{\beta HSWME}(U) = \left(\begin{array}{c}
     hu_m  \\
     hu_m^2 + g\frac{h^2}{2}+\frac{1}{3}h\alpha_{1}^{2} \\
     2hu_m\alpha_1 \\
     \frac{2}{3}h\alpha_{1}^{2} \\
     0 \\
     \vdots \\
     0
\end{array}\right),$$
the non-conservative matrix is given by 
$$B^{HSWME}(U) = \begin{pmatrix}
    0 & 0 & & & &\\
    0 & 0 & & & &  \\
     & & -u_m & \frac{3}{5}\alpha_{1}& & & \\
     & & -\alpha_1 & u_m & \frac{4}{7}\alpha_1 & \\
     & & & \frac{2}{5}\alpha_1 & \ddots & \ddots & \\
     & & & & \ddots & \ddots & \frac{N+1}{2N+1}\alpha_1\\
     & & & & & \frac{N-1}{2N-1}\alpha_1 & u_m 
    \end{pmatrix},
$$
$$B^{\beta HSWME}(U) = \begin{pmatrix}
    0 & 0 & & & &\\
    0 & 0 & & & &  \\
     & & -u_m & \frac{3}{5}\alpha_{1}& & & \\
     & & -\alpha_1 & u_m & \frac{4}{7}\alpha_1 & \\
     & & & \frac{2}{5}\alpha_1 & \ddots & \ddots & \\
     & & & & \ddots & \ddots & \frac{N+1}{2N+1}\alpha_1\\
     & & & & & \beta_{N} +\frac{N-1}{2N-1}\alpha_1 & u_m 
    \end{pmatrix},
$$
with $\beta_{N}= \frac{N^2 -N}{2N^2 + N - 1} \alpha_1$ the parameter of the $\beta$-HSWME model. 
The source term $S$ is the same as for the SWLME model. 

The respective terms for the generalized Roe scheme from \eqref{A_Roe}, \eqref{J_Roe}, \eqref{B_Roe} and \eqref{S_Roe} can be obtained by:
$$J_{i+\frac{1}{2}}^{HSWME} = J_{i+\frac{1}{2}}^{\beta HSWME} = \frac{\partial F}{\partial U}(h_{R}, u_{m,R}, \alpha_{1,R}, ..., \alpha_{N,R}),$$
using \eqref{J_Roe_r}, and 
$$ B_{i+\frac{1}{2}}^{HSWME} = B^{HSWME}(u_{m,b}, \alpha_{1,b}), \quad B^{\beta HSWME}_{i+\frac{1}{2}} = B^{\beta HSWME}(u_{m,b}, \alpha_{1,b}),$$
where the Roe averages are given by 
$$
    u_{m,b} = \begin{cases}
	\frac{h_{r}^{2}u_{m,r}+h_{l}^{2}u_{m,l}+h_{l}h_{r}\left[(u_{m,l}-u_{m,r})log\left(\frac{h_{r}}{h_{l}}\right)-(u_{m,r}+u_{m,l})\right]}{(h_{r}-h_{l})^{2}} & \text{if} \ \ h_{r} \neq h_{l}, \\
	\frac{u_{m,r} + u_{m,l}}{2} & \text{if} \ \ h_{r} = h_{l},
	\end{cases}
	$$
and
$$
    \alpha_{1,b} = \begin{cases}
	\frac{h_{r}^{2}\alpha_{1,r}+h_{l}^{2}\alpha_{1,l}+h_{l}h_{r}\left[(\alpha_{1,l}-\alpha_{1,r})log\left(\frac{h_{r}}{h_{l}}\right)-(\alpha_{1,r}+\alpha_{1,l})\right]}{(h_{r}-h_{l})^{2}} & \text{if} \ \ h_{r} \neq h_{l}, \\
	\frac{\alpha_{1,r} + \alpha_{1,l}}{2} & \text{if} \ \ h_{r} = h_{l}.
	\end{cases}
$$

\bibliographystyle{plain}
\bibliography{references}

\end{document}